\documentclass[11pt]{amsart}
\usepackage[margin=3cm]{geometry}
\usepackage{amsmath,amssymb,amscd}
\usepackage{amsthm}
\usepackage{ae} 
\usepackage{enumitem}
\usepackage{xcolor, graphicx}
\usepackage{subcaption}
\usepackage{tikz}
\usepackage[normalem]{ulem}
\usepackage{array}

\DeclareMathOperator{\Orb}{Orb}
\DeclareMathOperator{\Int}{Int}
\DeclareMathOperator{\End}{End}
\newcommand{\sa}{s\alpha}
\newcommand{\SA}{\mathrm{SA}}
\newcommand{\Per}{\mathrm{Per}}

\newcommand{\Rec}{\mathrm{Rec}}
\newcommand{\Bas}{\mathrm{Bas}}
\newcommand{\Sol}{\mathrm{Sol}}
\newcommand{\saBasin}{s\alpha\mathrm{Basin}}
\newcommand{\N}{\mathbb{N}}

\theoremstyle{plain}
\newtheorem{theorem}{Theorem}
\newtheorem{lemma}[theorem]{Lemma}
\newtheorem{proposition}[theorem]{Proposition}
\newtheorem{corollary}[theorem]{Corollary}
\theoremstyle{definition}
\newtheorem{definition}[theorem]{Definition}
\newtheorem{example}[theorem]{Example}
\newtheorem{remark}[theorem]{Remark}
\newtheorem{problem}[theorem]{Problem}
\newtheorem{conjecture}[theorem]{Conjecture}
\title{On backward attractors of interval maps}
\author{Jana Hant\'{a}kov\'{a}}
\address{Mathematical Institute of the Silesian University in Opava, Na Rybni\v{c}ku 1, 74601, Opava, Czech Republic}
\address{Faculty of Applied Mathematics, AGH University of Science and Technology, al. Mickiewicza 30, 30-059 Krak\'{o}w, Poland}
\email{jana.hantakova@math.slu.cz}
\author{Samuel Roth}
\address{Mathematical Institute of the Silesian University in Opava, Na Rybni\v{c}ku 1, 74601, Opava, Czech Republic}
\email{samuel.roth@math.slu.cz}

\makeatletter
\@namedef{subjclassname@2020}{%
  \textup{2020} Mathematics Subject Classification}
\makeatother
\thanks{The research was supported by RVO, Czech Republic funding for I\v{C}47813059. The first author was supported by the Marie Sk\l odowska-Curie grant agreement No 883748 from the European Union's Horizon 2020 research and innovation programme.}

\let\oldtocsection=\tocsection
\let\oldtocsubsection=\tocsubsection
\renewcommand{\tocsection}[2]{\hspace{0em}\oldtocsection{#1}{#2}}
\renewcommand{\tocsubsection}[2]{\hspace{2em}\oldtocsubsection{#1}{#2}}

\begin{document}

\begin{abstract}
Special $\alpha$-limit sets ($\sa$-limit sets) combine together all accumulation points of all backward orbit branches of a point $x$ under a noninvertible map. The most important question about them is whether or not they are closed. We challenge the notion of $\sa$-limit sets as backward attractors for interval maps by showing that they need not be closed. This disproves a conjecture by Kolyada, Misiurewicz, and Snoha. We give a criterion in terms of Xiong's attracting center that completely characterizes which interval maps have all $\sa$-limit sets closed, and we show that our criterion is satisfied in the piecewise monotone case. We apply Blokh's models of solenoidal and basic $\omega$-limit sets to solve four additional conjectures by Kolyada, Misiurewicz, and Snoha relating topological properties of $\sa$-limit sets to the dynamics within them. For example, we show that the isolated points in a $\sa$-limit set of an interval map are always periodic, the non-degenerate components are the union of one or two transitive cycles of intervals, and the rest of the $\sa$-limit set is nowhere dense. Moreover, we show that $\sa$-limit sets in the interval are always both $F_\sigma$ and $G_\delta$. Finally, since $\sa$-limit sets need not be closed, we propose a new notion of $\beta$-limit sets to serve as backward attractors. The $\beta$-limit set of $x$ is the smallest closed set to which all backward orbit branches of $x$ converge, and it coincides with the closure of the $\sa$-limit set. At the end of the paper we suggest several new problems about backward attractors.
\end{abstract}

\maketitle

\tableofcontents

\section{Introduction}\label{sec:introduction}
Let a \emph{discrete dynamical system} be defined as an ordered pair $(X,f)$ where $X$ is a compact metric space and $f$ is a continuous map acting on $X$. To understand the dynamical properties of such a system it is necessary to analyze the behavior of the trajectories of any point $x\in X$ under the iteration of $f$. Limit sets of trajectories are a helpful tool for this purpose since they can be used to understand the long term behavior of the dynamical system.

The \emph{$\omega$-limit sets} ($\omega(x)$ for short), i.e. the sets of limit points of forward trajectories, were deeply studied by many authors. For instance, one can ask for a criterion which determines whether a given closed invariant subset of $X$ is an $\omega$-limit set of some point $x\in X$. The question is very hard in general, however the answer for $\omega$-limit sets of a continuous map acting on the compact interval was provided by Blokh et al. in~\cite{BBHS}.  A closely related question is that of characterizing all those dynamical systems which may occur as restrictions of some system to one of its $\omega$-limit sets. These abstract $\omega$-limit sets were studied by Bowen \cite{Bowen} and Dowker and Frielander \cite{Dowker}. It was also proved that each $\omega$-limit set of a continuous map of the interval is contained in a maximal one by Sharkovsky \cite{Sharkovsky}.

\emph{$\alpha$-limit sets} ($\alpha(x)$ for short) were introduced as a dual concept to $\omega$-limit sets and they should be regarded as the source of the trajectory of a point. While for invertible maps $\alpha$-limit sets are well defined, for noninvertible maps there are many possibilities how to construct the limit of the backward trajectory. One possibility is to take as an $\alpha$-limit set the set of all accumulation points in $X$ of the preimage sets $f^{-n}(\{x\})$. This approach was used by Coven and Nitecki \cite{CovNit}, who showed that for an interval map, a point $x$ is nonwandering if and only if $x\in\alpha(x)$. A recent result about $\alpha$-limit sets of unimodal interval maps is due to Cui and Ding \cite{CuiDing}. Another approach used by Balibrea et al. \cite{Bal} proposes instead of looking at all possible preimages to pick one backward branch and check accumulation points of this sequence. The union of the sets of accumulation points over all backward branches of the map was called a \emph{special $\alpha$-limit set} ($s\alpha(x)$ for short) by Hero \cite{Hero}.

While $\alpha$-limit sets and special $\alpha$-limit sets seem to be similar to $\omega$-limit sets, they were not much explored so far. The reason for this is that they may have very rich structure, and also it is very hard to control the dynamics backward. For instance, it is clear that $\alpha$-limit sets or $\omega$-limit sets are always closed, but the situation of special $\alpha$-limit sets is unclear. By definition \cite{Hero}, those sets are in general uncountable unions of closed sets, so a priori their structure may be very complicated. A recent study by Kolyada et al.~\cite{KMS} provided some answers. In particular, they showed that a special $\alpha$-limit set need not be closed in the general setting. They investigated special $\alpha$-limit sets of interval maps and proved that for interval maps with a closed set of periodic points, every special $\alpha$-limit set has to be closed. This result led to the following conjecture:
\begin{conjecture}\cite{KMS}
	For all continuous maps of the unit interval all special $\alpha$-limit sets are closed.
\end{conjecture}
We disprove the conjecture by showing a counterexample of an interval map with a special $\alpha$-limit set which is not closed and give the properties of interval maps that determine if all special $\alpha$-limit sets are closed in Sections~\ref{sec:not-closed}. On the other hand, we show that for all continuous maps of the unit interval all special $\alpha$-limit sets are both $F_\sigma$ and $G_\delta$. We give further topological properties of special $\alpha$-limit sets of interval maps. If $\sa(x)$ is not closed, then it is uncountable and nowhere dense. If $\sa(x)$ is closed, then it is the union of a nowhere dense set and finitely many (perhaps zero) closed intervals, and in Section~\ref{sec:int-nwd} we prove some amount of transitivity of $f$ on those intervals. Since $\sa$-limit sets need not be closed, we propose a new notion of $\beta$-limit sets to serve as backward attractors in Section~\ref{sec:open-problems}. The $\beta$-limit set of $x$ is the smallest closed set to which all backward orbit branches of $x$ converge, and it coincides with the closure of the $\sa$-limit set.

Kolyada et al.~also made the following conjecture.
\begin{conjecture}\cite{KMS}
The isolated points in a special $\alpha$-limit set for an interval map are always periodic.
\end{conjecture}

 We verify this conjecture in Section~\ref{sec:iso-pts}. We also show that a countable special $\alpha$-limit set for an interval map is a union of periodic orbits. These results are opposite to the case of $\omega$-limit sets. The $\omega$-limit sets of a general dynamical system do not posses any periodic isolated points unless $\omega(x)$ is a single periodic orbit~\cite{Shar65}. 

The authors of \cite{KMS} also investigated the properties of special $\alpha$-limit sets of \emph{transitive} interval maps and stated the following conjecture:
\begin{conjecture}\cite{KMS}\label{conj:3}
Let $f:[0,1]\to[0,1]$ be a continuous map and $x,y\in[0,1]$.
	\begin{itemize}
	\item If $x\neq y$ and $s\alpha(x)=s\alpha(y)=[0,1]$, then $f$ is transitive.		
	\item If $s\alpha(x)=[0,1]$ then either $f$ is transitive or there is $c\in (0,1)$ such that $f|_{[0,c]}$ and $f|_{[c,1]}$ are transitive.
	\end{itemize}
\end{conjecture}
We slightly correct these conjectures in Section~\ref{sec:trans-full}. We show that $f$ is transitive if there are three distinct points $x, y, z \in [0,1]$ with $s\alpha(x)=s\alpha(y)=\sa(z)=[0,1]$ . If $f$ has one or two points with special $\alpha$-limit sets equal to $[0,1]$, but not more, then $[0,1]$ is the union of two transitive cycles of intervals.

It is known that if two $\omega$-limit sets of an interval map contain a common open set, then they are equal. The last conjecture in \cite{KMS} suggested that a similar property holds for special $\alpha$-limit sets:
\begin{conjecture}\cite{KMS}\label{conj:4}
Let $f$ be a continuous map $[0,1]\rightarrow [0,1]$ and $x,y\in[0,1]$. If $Int(s\alpha(x)\cap s\alpha(y))\neq\emptyset$ then $s\alpha(x)=s\alpha(y)$.	
\end{conjecture}
We correct this conjecture by showing that at most three distinct special $\alpha$-limit sets of $f$ can contain a given nonempty open set in Section~\ref{sec:common}.

The paper is organised as follows. Sections~\ref{sec:introduction} and~\ref{sec:terminology} are introductory. Section~\ref{sec:omega} investigates the relation of maximal $\omega$-limit sets to special $\alpha$-limit sets and provides tools necessary for proving the main results. It also contains a simple example showing that, unlike $\omega$-limit sets, the special $\alpha$-limit sets of an interval map need not be contained in maximal ones. Section~\ref{sec:general} is devoted to the above mentioned results on various properties of special $\alpha$-limit sets of interval maps. Section~\ref{sec:not-closed} studies properties of special $\alpha$-limit sets which are not closed. The paper closes with open problems and related questions in Section~\ref{sec:open-problems}.

\section{Terminology}\label{sec:terminology}
Let $X$ be a compact metric space and $f:X\to X$ a continuous map. A sequence $\{x_n\}_{n=0}^{\infty}$ is called
\begin{itemize}
\item the \emph{forward orbit} of a point $x$ if $f^n(x)=x_n \text{ for all }n\geq0$,
\item a \emph{preimage sequence} of a point $x$ if $f^n(x_n)=x \text{ for all }n\geq0$,
\item a \emph{backward orbit branch} of a point $x$ if $x_0=x$ and $f(x_{n+1})=x_{n}$ for all $n\geq0$.
\end{itemize}
A point $y$ belongs to the \emph{$\omega$-limit set of a point $x$}, denoted by $\omega(x)$, if and only if the forward orbit of $x$ has a subsequence $\{x_{n_i}\}_{i=0}^{\infty}$ such that $x_{n_i}\to y$. A point $y$ belongs to the \emph{$\alpha$-limit set of a point $x$}, denoted by $\alpha(x)$, if and only if some preimage sequence of $x$ has a subsequence $\{x_{n_i}\}_{i=0}^{\infty}$ such that $x_{n_i}\to y$. And a point $y$ belongs to the \emph{special $\alpha$-limit set of a point $x$}, also written as the \emph{$\sa$-limit set} and denoted $\sa(x)$, if and only if some backward orbit branch of $x$ has a subsequence $\{x_{n_i}\}_{i=0}^{\infty}$ such that $x_{n_i}\to y$. If we wish to emphasize the map, we will write $\omega(x,f)$, $\alpha(x,f)$ and $\sa(x,f)$.

To summarize, the $\omega$, $\alpha$, and $\sa$-limit sets of a point $x$ are defined as follows. The set $\omega(x)$ is the set of all accumulation points of its forward orbit and $\alpha(x)$ (resp. $\sa(x)$) is the set of all accumulation points of all its preimage sequences (resp. of all its backward orbit branches).

Let $T:X\rightarrow X$ and $F:Y\rightarrow Y$ be continuous maps of compact metric spaces. If there is a surjective map $\phi:X\rightarrow Y$ such that $\phi\circ T=F\circ \phi$ then it is said that $\phi$ semiconjugates $T$ to $F$ and $\phi$ is a \emph{semiconjugacy}.

Let $f:X \to X$ be a continuous map. A set $A\subset[0,1]$ is \emph{invariant} if $f(A)\subset A$. The forward orbit of a point, regarded as a subset of $X$ rather than a sequence, will be denoted by $\Orb(x)=\{f^n(x)~:~n\geq0\}$. The \emph{forward orbit of a set} is $\Orb(A)=\bigcup\{f^n(A)~:~n\geq 0\}$. We call $f$ \emph{transitive} if for any two nonempty open subsets $U, V \subset X$ there is $n\geq0$ such that $f^n(U)\cap V\neq\emptyset$. We call $f$ \emph{topologically mixing} if for any two nonempty open subsets $U, V\subset X$ there is an integer $N\geq 0$ such that $f^n(U)\cap V\neq \emptyset$ for all $n\geq N$.

Now let $f:[0,1]\to[0,1]$ be an interval map. We write $\Per(f)$, $\Rec(f)$, and $\Omega(f)$ for the sets of periodic points, recurrent points, and non-wandering points of $f$, respectively. We say that $x$ is \emph{preperiodic} if $x\notin\Per(f)$ but $f^n(x)\in\Per(f)$ for some $n\geq 1$. We write $\Lambda^1(f)=\bigcup_{x\in[0,1]}\omega(x)$ for the union of all $\omega$-limit sets of $f$ and $\SA(f)=\bigcup_{x\in[0,1]}\sa(x)$ for the union of all $\sa$-limit sets. Following~\cite{Xiong} we define the \emph{attracting center} of $f$ as $\Lambda^2(f)=\bigcup_{x\in\Lambda^1(f)} \omega(x)$. The \emph{Birkhoff center} of $f$ is the closure of the set of recurrent points $\overline{\Rec(f)}$ and coincides with $\overline{\Per(f)}$~\cite{CovHed}. If the map $f$ is clear from the context, we may drop it from the notation. The relation of these sets is given by the following summary theorem from the works of Hero and Xiong~\cite{Hero,Xiong}.
\begin{theorem}\label{th:structure}\cite{Hero,Xiong}
For any continuous interval map $f:[0,1]\to[0,1]$, we have
\begin{equation*}
\Per \subset \Rec \subset \Lambda^2 = \SA \subset \overline{\Rec} \subset \Lambda^1 \subset \Omega.
\end{equation*}
\end{theorem}

If $K\subset[0,1]$ is a non-degenerate closed interval such that the sets $K, f(K), \ldots, f^{k-1}(K)$ are pairwise disjoint and $f^k(K)=K$, then we call the set $M=\Orb(K)$ a \emph{cycle of intervals} and the \emph{period} of this cycle is $k$. We may also call $K$ an \emph{$n$-periodic interval}. If $f|_M$ is transitive then we call $M$ a \emph{transitive cycle} for $f$.

\section{Maximal $\omega$-limit sets and their relation to special $\alpha$-limit sets}\label{sec:omega}

%

An important property of the $\omega$-limit sets of an interval map $f$ is that each $\omega$-limit set is contained in a maximal one. These maximal $\omega$-limit sets come in three types: periodic orbits, basic sets, and solenoidal $\omega$-limit sets.

A \emph{solenoidal $\omega$-limit set} is  a maximal $\omega$-limit set which contains no periodic points. Any solenoidal $\omega$-limit set is uncountable and is contained in a nested sequence $\Orb(I_0) \supset \Orb(I_1) \supset \cdots$ of cycles of intervals with periods tending to infinity, also known as a \emph{generating sequence}~\cite[Assertion 4.2]{Blokh}. Here is the theorem relating $\sa$-limit sets to solenoidal $\omega$-limit sets; the proof is given in Section~\ref{sec:solenoidal}.

\begin{theorem}[Solenoidal Sets]\label{th:solenoidal} Let $\Orb(I_0) \supset \Orb(I_1) \supset \cdots$ be a nested sequence of cycles of intervals for the interval map $f$ with periods tending to infinity. Let $Q=\bigcap\Orb(I_n)$ and $S=Q\cap\Rec(f)$.
\begin{enumerate}
\item If $\alpha(y) \cap Q \neq \emptyset$, then $y\in Q$.
\item If $y\in Q$, then $\sa(y)\supset S$ and $\sa(y)\cap Q=S$.
\end{enumerate}
\end{theorem}

A \emph{basic set} is an $\omega$-limit set which is infinite, maximal among $\omega$-limit sets, and contains some periodic point. If $B$ is a basic set then with respect to inclusion there is a minimal cycle of intervals $M$ which contains it, and $B$ may be characterized as the set of those points $x\in M$ such that $\overline{\Orb(U)}=M$ for every relative neighborhood $U$ of $x$ in $M$, see~\cite{Blokh}. Conversely, if $M$ is a cycle of intervals for $f$, then we will write
\begin{equation*}
B(M,f)=\{x\in M: \text{for any relative neighborhood $U$ of $x$ in $M$ we have } \overline{\Orb(U)}=M\},
\end{equation*}
and if this set is infinite, then it is a basic set~\cite{Blokh}. Here is the theorem relating $\sa$-limit sets to basic sets; the proof is given in Section~\ref{sec:basic}.

\begin{theorem}[Basic Sets]\label{th:basic} Let $f$ be an interval map and fix $y\in[0,1]$.
\begin{enumerate}
\item If $\alpha(y)$ contains an infinite subset of a basic set $B=B(M,f)$, then $y\in M$ and $\sa(y)\supset B$.
\item If $\sa(y)$ contains a preperiodic point $x$, then there is a basic set $B=B(M,f)$ such that $x\in B\subset \sa(y)$.
\end{enumerate}
\end{theorem}
The sharpness of the second claim of Theorem~\ref{th:basic} is illustrated in Figure ~\ref{fig:prep}. The first map has two basic sets $B([0,1],f)$ and $B(M,f)$, where $M$ is the invariant middle interval. The $\sa$-limit set of 1 does not contain the basic set $B(M,f)$ although it includes the left endpoint of $M$, which is preperiodic. The second map shows that we cannot weaken the assumption to $\alpha(y)$. The $\alpha$-limit set of 1 includes the preperiodic endpoint of $M$ but $\sa(1)$ does not contain any basic set.

\begin{figure}[htb!!]

\begin{tikzpicture}
\draw (0,0) rectangle (3,3);
\draw (1,1) rectangle (2,2);
\draw (0,0) -- (3,3);
\clip (0,0) rectangle (3,3);
\draw[ultra thick] (0,0) -- (0.5,3) -- (1,2) -- (1.5,1) -- (2,2) -- (2.5,3) -- (3,0);
\end{tikzpicture}
\hspace{.2\textwidth}
\begin{tikzpicture}
\draw (0,0) rectangle (3,3);
\draw (1,1) rectangle (2,2);
\draw (0,0) -- (3,3);
\draw[ultra thick] (0,3) -- (1,2) -- (1.5,1) -- (2,2) -- (2.5,3) -- (3,3);
\end{tikzpicture}
\caption{A map where the $\sa$-limit set of 1 (respectively, the $\alpha$-limit set of 1) contains a preperiodic point from a basic set $B(M,f)$, but $\sa(1)\not\supset B(M,f)$.}\label{fig:prep}
\end{figure}
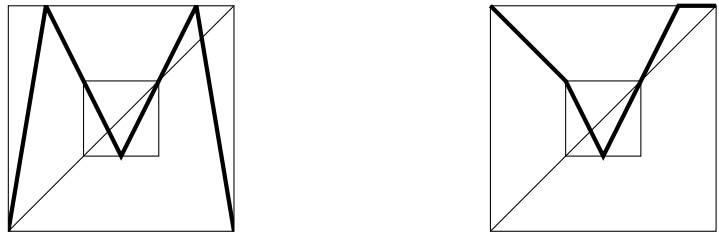

Periodic orbits may also be related to $\sa$-limit sets. The following result is one of the main theorems in~\cite{KMS}. Moreover, it holds for all periodic orbits of an interval map, even those which are not maximal $\omega$-limit sets.

\begin{theorem}\cite[Theorem 3.2]{KMS}\label{th:periodic}
Let $P$ be a periodic orbit for the interval map $f$. If $\alpha(y)\cap P\neq\emptyset$, then $\sa(y)\supset P$.
\end{theorem}

One additional observation is appropriate in this section. Unlike $\omega$-limit sets, the $\sa$-limit sets of an interval map need not be contained in maximal ones.

\begin{example}\label{ex:no-max-sa}
Fix two sequences of real numbers $1=a_1 > b_1 > a_2 > b_2 > \cdots$ both decreasing to $0$ and consider the ``connect-the-dots'' map $f:[0,1]\to[0,1]$ where
\begin{equation*}
f(0)=0,\, f(a_i)=a_i,\, f(b_i)=a_{i+2}, \quad (i=1,2,\ldots)
\end{equation*}
and $f$ is linear on all the intervals $[a_{i+1},b_i]$, $[b_i,a_i]$. The graph of such a function $f$ is shown in Figure~\ref{fig:no-max-sa}. The $\sa$-limit sets of this map are
\begin{equation*}
\sa(x)=\{a_1,\ldots,a_n\} \text{ for }x\in(a_{n+1},a_n] \text{ and } \sa(0)=\{0\}.
\end{equation*}
In particular, we get a strictly increasing sequence of $\sa$-limit sets and no $\sa$-limit set containing them all.

\begin{figure}[htb!!]
\includegraphics[width=.3\textwidth]{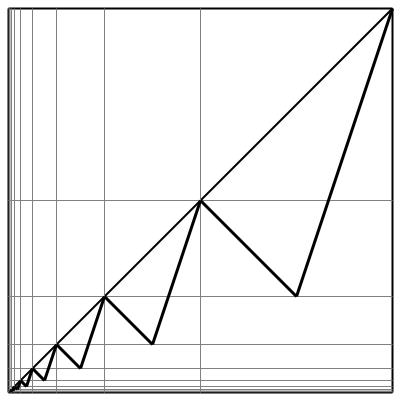}
\caption{A map with an increasing nested sequence of $\sa$-limit sets not contained in any maximal one.}\label{fig:no-max-sa}
\end{figure}
\end{example}

\subsection{Solenoidal sets}\label{sec:solenoidal}\strut\\
\indent This section is devoted to the proof of Theorem~\ref{th:solenoidal}.

We start with a broader definition of solenoidal sets, taken from~\cite{Blokh}. A \emph{generating sequence} is any nested sequence of cycles of intervals $\Orb(I_0)\supset\Orb(I_1)\supset\cdots$ for $f$ with periods tending to infinity. The intersection $Q=\bigcap_n\Orb(I_n)$ is automatically closed and \emph{strongly invariant}, i.e. $f(Q)=Q$, and any closed and strongly invariant subset $S$ of $Q$ (including $Q$ itself) will be called a \emph{solenoidal set}. Two examples described in~\cite{Blokh} which we will need later are
\begin{enumerate}
\item the set of all $\omega$-limit points in $Q$, denoted $S_\omega=S_\omega(Q)=Q\cap\Lambda^{1}(f)$, and
\item the set of all recurrent points in $Q$, denoted $S_{\Rec}=S_{\Rec}(Q)=Q\cap\Rec(f)$.
\end{enumerate}

Blokh showed that $Q$ contains a perfect set $S$ such that $S=\omega(x)$ for all $x\in Q$~\cite[Theorem 3.1]{Blokh}. Clearly $S=S_{\Rec}$. We refer to $S_{\Rec}$ as a \emph{minimal solenoidal set} both because it is the smallest solenoidal set in $Q$ with respect to inclusion, and because the mapping $f|_{S_{\Rec}}$ is minimal, i.e. all forward orbits are dense.

If $\omega(x)$ is a maximal $\omega$-limit set for $f$ and contains no periodic points (what Sharkovskii calls a maximal $\omega$-limit set of genus 1), then it is in fact a solenoidal set~\cite[Assertion 4.2]{Blokh}. Thus it has a generating sequence $\Orb(I_0)\supset\Orb(I_1)\supset \cdots$ of cycles of intervals and belongs to their intersection $Q$. If $Q'=\bigcap_n\Orb(I'_n)$ is formed from another generating sequence for $f$, then it is well known (and an easy exercise) that $Q$ and $Q'$ are either identical or disjoint. This means that given any solenoidal set $S$, there is a unique maximal solenoidal set $Q$ which contains it (so $Q$ is uniquely determined, even if the generating sequence is not).

One can use a translation in a zero-dimensional infinite group as a model for the map $f$ acting on a solenoidal set $Q=\bigcap \Orb(I_j)$. Let $D=\{m_j\}_{j=0}^{\infty}$ where $m_j$ is the period of $\Orb(I_j)$ and let $H(D)=\{(r_0,r_1,\ldots):r_{j+1}=r_j \: (\mathrm{mod}\,\, m_j),\text{ for all }j\geq 0\}$ where $r_j$ is an element of the group of residues modulo $m_j$, for every $j$. Denote by $\tau$ the translation in $H(D)$ by the element $(1,1,\ldots)$.

\begin{theorem}\cite[Theorem 3.1]{Blokh}\label{th:solenoid-model}
There exists a semiconjugacy $\phi:Q\rightarrow H(D)$ between $f|_Q$ and $\tau$ such that, for every $r\in H(D)$, the set $J=\phi^{-1}(r)$ is a connected component of $Q$ and it is either a singleton $J=\{a\}, a\in S_{\Rec}$, or an interval $J=[a,b],\:\emptyset\neq S_{\Rec}\cap J\subset S_{\omega}\cap J\subset\{a,b\}.$
\end{theorem}

%

One lemma which we will need several times throughout the paper is the following

\begin{lemma}\label{lem:inv-sa}
If $A$ is invariant for $f$ and $\alpha(x)\cap\Int(A)\neq\emptyset$, then $x\in A$. In particular, if $\sa(x)\cap\Int(A)\neq\emptyset$, then $x\in A$.
\end{lemma}
\begin{proof}
Choose $a\in\alpha(x)\cap\Int(A)$ and choose a neighborhood $U$ of $a$ contained in $A$. There is $n\in\mathbb{N}$ and a point $x_{-n}\in U$ such that $f^n(x_{-n})=x$. Since $U\subset A$ and $A$ is invariant, $x$ must belong to $A$. We get the same conclusion when $\sa(x)\cap\Int(A)\neq\emptyset$, because $\sa(x)\subset\alpha(x)$.
\end{proof}

Now we are ready to give the proof of Theorem~\ref{th:solenoidal}.

\begin{proof}[Proof of Theorem~\ref{th:solenoidal}]
(1):  Fix $z\in Q\cap\alpha(y)$ and let $S=S_{\Rec}=Q\cap\Rec(f)$ be the minimal solenoidal set in $Q$. Then by~\cite[Theorem 3.1]{Blokh} $S=\omega(z)$ and since $\alpha(y)$ is a closed invariant set it must contain $S$. In particular, $\alpha(y)$ contains infinitely many points from each cycle of intervals $\Orb(I_n)$, and so by Lemma~\ref{lem:inv-sa} $y\in \Orb(I_n)$, for all $n$. Therefore $y\in Q$.

(2): Fix $y\in Q$. Since $f(Q)=Q$ we can choose a backward orbit branch for $y$ which never leaves $Q$. Therefore it has an accumulation point in $Q$, and so $\sa(y)\cap Q\neq\emptyset$. Let $w\in \sa(y)\cap Q$. By~\cite[Theorem 3.1]{Blokh}, $\omega(w)$ is the minimal solenoidal set $S$. According to~\cite[Lemma 1]{Hero}, if $\sa(y)$ contains a point, then it contains the whole $\omega$-limit set of that point as well. Therefore $\sa(y) \supset S$ and $\sa(y)\cap Q\supset S$. To finish the proof it is enough to show the opposite inclusion $\sa(y)\cap Q\subset S$.

We will assume otherwise. Suppose there is a point $z\in \sa(y)\cap (Q\setminus S)$. Let $\phi(z)=r$, where $\phi$ is defined in Theorem $\ref{th:solenoid-model}$. Since $z\in\sa(y)\subset \SA(f)\subset \Lambda^1(f)$, we can assume $z\in \Lambda^1(f)\cap(Q\setminus S)=S_{\omega}\setminus S$. By Theorem $\ref{th:solenoid-model}$, $\phi^{-1}(r)$ has to be an interval and $z$ has to be one of its endpoints, say, the right endpoint, and $\phi^{-1}(r)=[x,z]$, $x\in S.$ Since $\phi$ is a semiconjugacy we have $f^i([x,z])\subset\phi^{-1}(\tau^i(r))$ for all $i\geq0$. But the intervals $\phi^{-1}(\tau^i(r))$ are pairwise disjoint. This shows that $[x,z]$ is a wandering interval.

{\bf Claim} $z\in\Int (\Orb(I_j)),$ for every $j\geq 0$.\\
We will assume otherwise. Let $K$ be the connected component of $\Orb(I_N)$, for some $N\geq0$, where $z$ is an endpoint of $K$. Let $v$ be a point such that $z\in\omega(v)$. By \cite[Theorem 3.1]{Blokh}, $S=\omega(z)\subset\omega(v)$, we have $\omega(v)\cap\Orb(I_N)$ infinite and necessarily $\Orb(v)\cap\Int(\Orb(I_N))\neq\emptyset$. This implies $f^k(v)\in\Orb(I_N)$ for all sufficiently large $k$. It follows that $\Orb(v)$ accumulates on $z$ from the interior of $K$ and we can find $k>0$ such that $f^k(v)\in(x,z)$. But $[x,z]$ is a wandering interval, so $\Orb(v)$ cannot accumulate on $z$ which contradicts $z\in\omega(v)$.\\

Let $\{y_n\}_{n=0}^\infty$ be a backward orbit branch of $y$ with a subsequence $\{y_{n_i}\}_{i=0}^\infty$ such that $\lim_{i\to\infty} y_{n_i}=z$. Since $z\in\Int(\Orb I_j)$ it follows from Lemma~\ref{lem:inv-sa} that $y_n\in\Orb(I_j)$ for all $j,n \geq 0$. Therefore $\{y_n\}_{n=0}^\infty \subset Q$. For every $n\geq 1$, denote $\phi(y_n)=r_n$. Then by Theorem $\ref{th:solenoid-model}$, $\phi^{-1}(r_n)$ are connected, pairwise disjoint sets, each containing an element $s_n\in S$. Since $y_n\in\phi^{-1}(r_n)$, we have $\lim_{i\to\infty} s_{n_i}=z$. But $S$ is a closed set and $z\notin S$, which is impossible. Therefore $\sa(y)\cap (Q\setminus S)=\emptyset$ and $\sa(y)\cap Q\subset S$.
\end{proof}

\begin{corollary}\label{cor:at-most-1}
A $\sa$-limit set contains at most one solenoidal set.
\end{corollary}
\begin{proof}
Let $\Orb(I_0) \supset \Orb(I_1) \supset \cdots$ and $\Orb(I'_0) \supset \Orb(I'_1) \supset \cdots$ be nested sequences of cycles of intervals generating two solenoidal sets $Q=\bigcap\Orb(I_n)$ and $Q'=\bigcap_n\Orb(I'_n)$. If $\sa(y)\cap Q\neq \emptyset$ and $\sa(y)\cap Q'\neq \emptyset$ then, by Theorem $\ref{th:solenoidal}$, $y\in Q\cap Q'$. Since two solenoidal sets $Q$ and $Q'$ are either identical or disjoint we have $Q=Q'$. Then the only solenoidal set contained in $\sa(y)$ is $S=Q\cap\Rec(f)$.
\end{proof}

\subsection{Basic sets}\label{sec:basic}\strut\\\indent
This section is devoted to the proof of Theorem~\ref{th:basic}.\\
Let $f$ be a continuous map acting on an interval $I$. We say that an endpoint $y$ of $I$ is \emph{accessible} if there is $x\in\Int (I)$ and $n\in\mathbb{N}$ such that $f^n(x)=y$. If $y$ is not accessible, then it is called \emph{non-accessible}. The following Proposition is derived from~\cite[Proposition 2.8]{Ruette}.
\begin{proposition}\cite{Ruette}\label{mixing}
Let $f$ be a topologically mixing map acting on an interval $I$. Than for every $\epsilon>0$ and every $x\in I$ such that the interval $[x-\epsilon,x+\epsilon]\subset I$ does not contain a non-accessible endpoint and for every non degenerate interval $U\subset I$, there exists an integer $N$ such that $f^n(U)\supset [x-\epsilon,x+\epsilon]$, for all $n\geq N$.
\end{proposition}
 An \emph{$m$-periodic transitive map of a cycle of intervals} is a transitive map $g:M\to M$, where $M\subset\mathbb{R}$ is a finite union of pairwise disjoint compact intervals $I$, $g(I)$, \ldots, $g^{m-1}(I)$, and $g^m(I)=I$. We write $\End(M)$ for the endpoints of the connected components of $M$ and refer to these points simply as \emph{endpoints of $M$}. The set of \emph{exceptional points of $g$} is defined
\begin{equation*}
E := M \setminus \bigcap_U \bigcup_{n=1}^\infty g^n(U),
\end{equation*}
where $U$ ranges over all relatively open nonempty subsets of $M$. It is known that $E$ is finite; it can contain some endpoints and at most one non-endpoint from each component of $M$.  If $g$ is topologically mixing, then $E=\bigcup_{i=0}^{m-1}E_i$, where $E_i$ is the set of non-accessible endpoints of $g^m|_{g^i(I)}$ by Proposition \ref{mixing}. If $g$ is transitive but not mixing, then by~\cite[Proposition 2.16]{Ruette} there is an $m$-periodic orbit $c_0,c_1,\ldots,c_{m-1}$ of points such that $c_i\in\Int(g^i(I))$, $g^{2m}|_{[a_i,c_i]}$ and $g^{2m}|_{[c_i,b_i]}$ are topologically mixing interval maps, where $g^i(I)=[a_i,b_i]$, for $i=0,\ldots m-1$. Then $E=\bigcup_{i=0}^{m-1}E_i$, where $E_i$ is the union of the sets of non-accessible endpoints of $g^{2m}|_{[a_i,c_i]}$ and $g^{2m}|_{[c_i,b_i]}$.

By~\cite[Lemma 2.32]{Ruette}, every point in $E$ is periodic and therefore $g(E)=E$. By the definition of non-accessible points, $g^{-1}(E)\cap M=E$.

We use $m$-periodic transitive maps of cycles of intervals as models for maps acting on basic sets. We recall again Blokh's definition. If $M$ is a cycle of intervals for the map $f:[0,1]\to[0,1]$, then we set
\begin{equation*}
B(M,f)=\{x\in M: \text{for any relative neighborhood $U$ of $x$ in $M$ we have } \overline{\Orb(U)}=M\},
\end{equation*}
and if $B(M,f)$ is infinite, then it is a basic set for $f$.

\begin{theorem}\cite[Theorem 4.1]{Blokh}\label{th:basic-model}
Let $I$ be an $m$-periodic interval for $f$, $M=\Orb(I)$ and $B=B(M,f)$ be a basic set. Then there is a transitive $m$-periodic map $g:M'\rightarrow M'$ and a monotone map $\phi:M\rightarrow M'$ such that $\phi$ semiconjugates $f|_M$ to $g$ and $\phi(B)=M'$. Moreover, for any $x\in M'$, $1\leq \#\{\phi^{-1}(x)\cap B\}\leq 2$ and $\Int (\phi^{-1}(x))\cap B=\emptyset$, and so $\phi^{-1}(x)\cap B\subset \partial \phi^{-1}(x)$. Furthermore, $B$ is a perfect set. 
\end{theorem}

\begin{lemma}\label{lem:basic}
Let $B$ be a basic set, $M$ the smallest cycle of intervals for $f$ which contains $B$, and $\phi:(M,f)\to(M',g)$ the semiconjugacy to the $m$-periodic transitive map $g$ given by Theorem~\ref{th:basic-model}. Let $E$ be the set of exceptional points of the map $g$ acting on $M'$. Suppose $y\in M$ and $\phi(y) \not\in E\cup \End(M')$. Then $\sa(y)\supset B$.

\end{lemma}
\begin{proof}
Let $x\in B$. There is $\epsilon>0$ such that $\phi|_{(x,x+\epsilon)}$ is not constant and $\phi((x,x+\epsilon))\cap\End(M')=\emptyset$ or $\phi|_{(x-\epsilon,x)}$ is not constant and $\phi((x-\epsilon,x))\cap\End(M')=\emptyset$. Otherwise $x$ has a neighborhood $N$ such that $\phi|_N$ is constant which is in a contradiction with $x\in B$ by Theorem \ref{th:basic-model}. We can assume  $\phi|_{(x,x+\epsilon)}$ is not constant and $\phi((x,x+\epsilon))\cap \End(M')=\emptyset$, and denote $V=(x,x+\epsilon)$. Then $U=\phi(V)$ is a non-degenerate interval in $M'$. Since $\phi(y)\notin E\cup \End(M')$, there is $\delta>0$ such that $[\phi(y)-\delta,\phi(y)+\delta]\subset M'$ and $[\phi(y)-\delta,\phi(y)+\delta]\cap E=\emptyset$. The set $E$ was defined as a union of non-accessible endpoints of a topologically mixing map $g^m$ (resp. $g^{2m})$, therefore we can use Proposition \ref{mixing} for the map $g^m$ (resp. $g^{2m})$ acting on $M'$. There is an $N>0$ such that $g^N(U)\supset [\phi(y)-\delta,\phi(y)+\delta]$. 
 But $\phi(f^N(V))=g^N(U)$, so $\phi(y)$ is in $\phi(f^N(V))$. Since $\phi$ is monotone this means either $y\in f^N(V)$ or $\phi(y)$ is an endpoint of the interval $\phi(f^N(V))$. But we have seen that it is not an endpoint. Therefore $y\in f^N(V)$ and we can find $y_1\in V$ and $N_1=N$ such that $f^{N_1}(y_1)=y.$
Notice that $\phi(y_1)\notin E$ since $g^{N_1}(\phi(y_1))=\phi(y)\notin E$ and $g^{-N_1}(E)\cap M'=E$; and $\phi(y_1)\notin \End(M')$ since $y_1\in V$ and $\phi(V)\cap \End(M')=\emptyset$.  By the same procedure, we can find $y_2\in(x,x+\epsilon/2)\cap M$ and $N_2\in\mathbb{N}$ such that $f^{N_2}(y_2)=y_1$. By repeating this process we construct a sequence $\{y_n\}_{n=1}^{\infty}$ converging to $x$ which is a subsequence of a backward orbit branch of $y$. Since $x\in B$ was arbitrary, this shows $B\subset s\alpha(y)$.
\end{proof}
\begin{corollary}\label{cor:bs}
For every basic set $B$, there is $y\in B$ such that $\sa(y)\supset B$.
\end{corollary}
\begin{proof}
Let $M,\phi, g, M', E$ be as in the previous proof. Since the map $\phi|_B$ is at most 2-to-1 and $E$ is a finite set, there are uncountably many points $y\in B$ such that $\phi(y)\notin E\cup\End(M')$. The result follows by Lemma \ref{lem:basic}.
\end{proof}
Before we proceed to the proof of Theorem~\ref{th:basic} we need to recall the definition of a \emph{prolongation set} and its relation to basic sets. Let $M$ be a cycle of intervals. Let the side $T$ be either the left side $T=L$ or the right side $T=R$ of a point $x\in M$ and $W_T(x)$ be a \emph{one-sided neighborhood} of $x$ from the $T$-hand side, i.e. $W_T(x)$ contains for some $\epsilon>0$ the interval $(x,x+\epsilon)$ (resp. $(x-\epsilon,x)$) when $T=R$ (resp. $T=L$). We do not consider the side $T=R$ (resp. $T=L$) when $x$ is a right endpoint (resp. left endpoint) of a component of $M$. Now let
$$P^T_M(x)=\bigcap_{W_T(x)}\bigcap_{n\geq 0}\overline{\bigcup_{i\geq n}f^i(W_T(x)\cap M)},$$
where the intersection is taken over the family of all one-sided neighborhoods $W_T(x)$ of $x$. We will write $P^T(x)$ instead of $P^T_{[0,1]}(x)$. The following Lemma \ref{lem:prol} and Lemma \ref{lem:prol2} about properties of prolongation sets are taken from~\cite{Blokh}.
\begin{lemma}\label{lem:prol}\cite{Blokh}
Let $x\in[0,1]$. Then $P^T(x)$ is a closed invariant set and only one of the following possibilities holds:
\begin{itemize}
\item There exists a wandering interval $W_T(x)$ with pairwise disjoint forward images and $P^T(x)=\omega(x)$. 
\item There exists a periodic point $p$ such that $P^T(x)=\Orb(p)$.
\item There exists a solenoidal set $Q$ such that $P^T(x)=Q$.
\item There exists a cycle of intervals $M$ such that $P^T(x)=M$.
\end{itemize}
\end{lemma}
There is a close relation between prolongation sets and basic sets. If $M$ is a cycle of intervals for $f$ then we define
\begin{equation*}
E(M,f)=\{x\in M:\text{ there is a side $T$ of $x$ such that }P^T_M(x)=M\},
\end{equation*} and, for $x\in E(M,f)$, we call this side $T$ a \emph{source side} of $x$.

\begin{lemma}\label{lem:prol2}\cite{Blokh}
Either the set $E(M,f)$ is a single periodic orbit, or it is infinite and $E(M,f)=B(M,f)$. In the latter case, for any $x\in B(M,f)$ with a source side $T$ and any one-sided neighborhood $W_T(x)$, we have $W_T(x)\cap B(M,f)\neq \emptyset$.
\end{lemma}

\begin{proof}[Proof of Theorem~\ref{th:basic}]
(1) Let $\phi$ and $g$ be the maps given in Theorem~\ref{th:basic-model} and $E$  be the set of exceptional points of the map $g$ acting on $M'$. Since $\phi|_B$ is an at most 2-to-1 map, $\phi(\alpha(y)\cap B)$ is an infinite subset of $M'$. But $E\cup \End(M')$ is a finite set, so we can find a point $z\in\alpha(y)\cap B$ such that $\phi(z)\notin E\cup \End(M')$ and therefore $z\notin \phi^{-1}(E\cup\End(M'))$. The set $\phi^{-1}(E\cup\End(M'))$ is a union of finitely many, possibly degenerate, closed intervals in $M$. Since $z\in (\alpha(y)\cap B )\setminus \phi^{-1}(E\cup\End(M'))$, there is a pre-image $y'\in M$ of $y$, $y=f^k(y')$, for some $k\geq 0$, and simultaneously $y'\notin \phi^{-1}(E\cup\End(M'))$, which implies $\phi(y')\notin E\cup\End(M')$. Then $y\in M$ by the invariance of $M$. By Lemma~\ref{lem:basic} applied to $y'$, $\sa(y')\supset B$. But the containment $\sa(y)\supset\sa(y')$ is clear from the definition of $\sa$-limit sets, and so $\sa(y)\supset B$.\\
(2) Let $\{y_i\}_{i=0}^{\infty}$ be the backward orbit branch of $y$ accumulating on $x$. Since $x$ is not a periodic point, it is not contained in $\{y_i\}_{i=0}^{\infty}$ more then one time and we can assume that $\{y_i\}_{i=0}^{\infty}$ accumulates on $x$ from one side $T$. Consider the prolongation set $P^T(x)$. Clearly $\{y_i\}_{i=0}^{\infty}\subset P^T(x)$. Since $P^T(x)$ is closed and invariant, $x$ and $\Orb(x)$ belong to $P^T(x)$, we see that $P^T(x)$ contains both periodic and non-periodic points. By Lemma \ref{lem:prol}, there is only one possibility $P^T(x)=M$, where $M$ is a cycle of intervals. The other possibilities are ruled out - $\Orb(p)$, where $p$ is a periodic point and $\omega(x)$, where $x$ is a preperiodic point, can not contain a non-periodic point; and a solenoidal set $Q$ can not contain a periodic point. Since $P^T(x)=M$ contains $\{y_i\}_{i=0}^{\infty}$ it must contain a $T$-sided neighborhood of $x$ and therefore $P^T_M(x)=P^T(x)=M$. Let $E(M,f)=\{z\in M:\text{ there is a side $S$ of $z$ such that }P^S_M(z)=M\}$. Since $x\in E(M,f)$ and $x$ is not periodic, by Lemma \ref{lem:prol2}, $E(M,f)=B(M,f)$ and $T$ is a source side of $x$ in $B(M,f)$.  Let $\phi$ and $g$ be the maps given in Theorem~\ref{th:basic-model} and $E$  be the set of exceptional points of the map $g$ acting on $M'$.  By Lemma \ref{lem:prol2}, $\phi$ is not constant on any $T$-sided neighborhood of $x$ and since $\{y_i\}_{i=0}^{\infty}$ accumulates on $x$ from the source side $T$, we can find a pre-image $y'\in M$ of $y$ such that $\phi(y')\notin E\cup \End(M')$. Then $\sa(y')\supset B$ by Lemma~\ref{lem:basic}, and $\sa(y)\supset\sa(y')$ since $y'$ is a preimage of $y$. We conclude that $\sa(y)\supset B$.

\end{proof}
We record here one corollary which we will need several times in the rest of the paper.

\begin{corollary}\label{cor:tc}
If $\sa(x)$ contains infinitely many points from a transitive cycle $M$, then $x\in M$ and $\sa(x)\supseteq M$.
\end{corollary}
\begin{proof}
In this case $M$ is itself a basic set, so we may apply Theorem~\ref{th:basic}.
\end{proof}

\section{General properties of special $\alpha$-limit sets for interval maps}\label{sec:general}

\subsection{Isolated points are periodic}\label{sec:iso-pts}\strut\\\indent
Unless an $\omega$-limit set is a single periodic orbit, its isolated points are never periodic~\cite{Shar65}. The opposite phenomenon holds for the $\sa$-limit sets of an interval map.

\begin{theorem}
Isolated points in a $\sa$-limit set for an interval map are periodic.
\end{theorem}
\begin{proof}Let $z\in s\alpha(y)$ such that $z$ is neither periodic nor preperiodic. Then $z$ is a point of an infinite maximal $\omega$-limit set, i.e.~a basic set or a solenoidal set. This follows from Blokh's Decomposition Theorem, that $\Lambda^1(f)$ is the union of periodic orbits, solenoidal sets and basic sets, and from Theorem~\ref{th:structure}, $z\in \SA(f)\subset\Lambda^1(f)$.  According to~\cite[Lemma 1]{Hero}, when $\sa(y)$ contains a point $z$, it contains its orbit $\Orb(z)$ as well. If $z$ is in a basic set $B$ then, $\Orb(z)\subset B\cap\sa(y)$ is infinite and by Theorem~\ref{th:basic} (1), $B\subset \sa(y)$. Then the point $z$ is not isolated in $\sa(y)$ since $B$ is a perfect set. If $z$ is in a solenoidal set $Q=\bigcap_n\Orb(I_n)$ then, by Theorem~\ref{th:solenoidal}, $\sa(y)\cap Q=S$. Again, the point $z$ is not isolated in $ \sa(y)$ since $S$ is a perfect set.

Let $z\in s\alpha(y)$ such that $z$ is a preperiodic point. By Theorem~\ref{th:basic} (2), there is a basic set $B$ such that $z\in B\subset \sa(y)$. Then the point $z$ is not isolated in $\sa(y)$ since $B$ is a perfect set. \end{proof}

In the previous proof, we have shown that if  a point $z\in s\alpha(y)$ is not periodic then $s\alpha(y)$ contains either a minimal solenoidal set $S$ or a basic set $B$. In both cases, $s\alpha(y)$ has to be uncountable. Therefore we have the following corollary. 
\begin{corollary}
A countable $\sa$-limit set for an interval map is a union of periodic orbits.
\end{corollary}

%

\subsection{The interior and the nowhere dense part of a special $\alpha$-limit set}\label{sec:int-nwd}\strut\\\indent
A well-known result by Sharkovsky says that each $\omega$-limit set of an interval map is either a transitive cycle of intervals or a closed nowhere dense set~\cite{Sharkovsky}. What can we say in this regard for $\sa$-limit sets of interval maps? When $\Int(\sa(x))$ is nonempty, Kolyada, Misiurewicz, and Snoha showed that $M=\overline{\Int{\sa(x)}}$ is a cycle of intervals containing $x$, see \cite[Proposition 3.6]{KMS}. We strengthen this result by showing that the non-degenerate components of $\sa(x)$ are in fact closed, and the rest of $\sa(x)$ is nowhere dense. We also get some amount of transitivity.

The following lemma is simple and we leave the proof to the reader.

\begin{lemma}\label{lem:cycle}
Let $M$ be a cycle of intervals for $f$ of period $k$ and let $K$ be any of its components. Then
\begin{enumerate}
\item $f|_M$ is transitive if and only if $f^k|_K$ is transitive.
\item If $L\subset K$ is a cycle of intervals for $f^k$, then $\bigcup_{i<k} f^i(L)$ is a cycle of intervals for $f$.
\item For $x\in K$ and $y\in K\setminus\End(K)$, we have $y\in\sa(x,f)$ if and only if $y\in\sa(x,f^k|_K)$.
\end{enumerate}
\end{lemma}

\begin{theorem}\label{th:M-nd}
A $\sa$-limit set for an interval map $f$ is either nowhere dense, or it is the union of a cycle of intervals $M$ for $f$ and a nowhere dense set. Moreover, $M$ is either a transitive cycle, or it is the union of two transitive cycles.
\end{theorem}

\begin{proof}
Consider a limit set $\sa(x)$ for an interval map $f:[0,1]\to[0,1]$. Let $M$ be the union of the non-degenerate components of $\overline{\sa(x)}$. If $M=\emptyset$ then $\sa(x)$ is nowhere dense. Otherwise $M$ must be a finite or countable union of closed intervals, and since $M$ contains the interior of the closure of the $\sa$-limit set we know that $\sa(x)\setminus M$ is nowhere dense.

Let $K$ be any component of $M$. By Theorem~\ref{th:structure} we have $K\subset\overline{\Per(f)}$, and therefore periodic points are dense in $K$. Let $n\geq 1$ be minimal such that $f^n(K)\cap K\neq\emptyset$. Since $f^n(K)$ is connected and $K$ is a component of the invariant set $\overline{\sa(x)}$, we know that $f^n(K)\subseteq K$. Since periodic points are dense in $K$ we must have $f^n(K)=K$. Therefore $\Orb(K)$ is a cycle of intervals, and by Lemma~\ref{lem:inv-sa} we get $x\in \Orb(K)$. Since this holds for every component $K$ of $M$, we must have $\Orb(K)=M$, i.e. $M$ is a cycle of intervals and $x\in M$.

From now on we take $K$ to be the component of $M$ containing $x$. Put $g=f^n|_K$. Then $g:K\to K$ is an interval map with a dense set of periodic points. There is a structure theorem for interval maps with a dense set of periodic points~\cite[Theorem 3.9]{Ruette}%
\footnote{In fact, \cite[Theorem 3.9]{Ruette} tells us that all the transitive cycles for $g$ have period at most 2, and the periodic orbits not contained in transitive cycles also have period at most $2$. Some of this extra information can be shown quite easily; it comes up again in our proof of Theorem~\ref{th:at-least-3}.}
which tells us that $K$ is a union of the transitive cycles and periodic orbits of $g$,
\begin{equation*}
K=\left(\bigcup \Big\{ L~:~L\text{ is a transitive cycle for $g$} \Big\} \right) \cup \Per(g)
\end{equation*}

By Lemma~\ref{lem:cycle}, $\sa(x,g)$ contains a dense subset of $K$. By Corollary~\ref{cor:tc} each transitive cycle $L\subseteq K$ for $g$ must contain $x$. Since transitive cycles have pairwise disjoint interiors, $g$ has at most two transitive cycles. If their union is not $K$, then $K$ must contain a non-degenerate interval of periodic points of $g$. But by an easy application of Lemma~\ref{lem:inv-sa}, no $\sa$-limit set can contain a dense subset of an interval of periodic points. Therefore $K$ is the union of one or two transitive cycles for $g$. By Lemma~\ref{lem:cycle}, $M$ is the union of one or two transitive cycles for $f$.

Finally, if $L$ is one of the (at most two) transitive cycles for $f$ that compose $M$, then by Corollary~\ref{cor:tc} we have $\sa(x)\supseteq L$. Therefore $M\subseteq \sa(x)$.
\end{proof}

\begin{remark}
If $\sa(x)$ contains a cycle of intervals $M$, then $\sa(x)$ is in fact a closed set, but we are not yet ready to prove this fact. See Theorem~\ref{th:not-closed} below.
\end{remark}


\subsection{Transitivity and points with $\sa(x)=[0,1]$}\label{sec:trans-full}\strut\\\indent
Let $f:[0,1]\to [0,1]$ be an interval map. We say that the point $x$ has a \emph{full} $\sa$-limit set if $\sa(x)=[0,1]$.  Kolyada, Misiurewicz, and Snoha proved that when $f$ is transitive, all points $x\in[0,1]$ with at most three exceptions have a full $\sa$-limit set. Conversely, they conjectured that if at least two points $x\in[0,1]$ have full $\sa$-limit sets, then $f$ is transitive. The conjecture was not quite right; the correct result is as follows:

\begin{theorem}\label{th:at-least-3}
An interval map $f:[0,1]\to[0,1]$ is transitive if $\sa(x_i)=[0,1]$ for at least three distinct points $x_1,x_2,x_3$.
\end{theorem}
\begin{proof}
Suppose that $f$ is not transitive. We will prove that at most two points have a full $\sa$-limit set. Suppose there is at least one point $x$ with $\sa(x)=[0,1]$. By Theorem~\ref{th:M-nd} there are two transitive cycles $L,L'$ for $f$ such that $[0,1]=L\cup L'$, and by Corollary~\ref{cor:tc} every point with a full $\sa$-limit set belongs to $L\cap L'$. We will show that the cardinality of $L\cap L'$ is at most two.

Let $A_1, A_2, \ldots, A_n$ be the components of $L$, numbered from left to right in $[0,1]$. Let $\sigma:\{1,\ldots,n\} \to \{1,\ldots,n\}$ be the cyclic permutation defined by $f(A_i)=A_{\sigma(i)}$. If $n\geq3$ then there must exist $i$ such that $|\sigma(i)-\sigma(i+1)|\geq2$, so there is some $A_j$ strictly between $A_{\sigma(i)}, A_{\sigma(i+1)}$. Let $B$ be the component of $L'$ between $A_i$ and $A_{i+1}$, see Figure~\ref{fig:proof}. By the intermediate value theorem, $f(B)\supset A_j$. This contradicts the invariance of $L'$. Therefore $L$ has at most 2 components. For the same reason $L'$ has at most two components. Moreover, $L, L'$ cannot both have 2 components; otherwise the middle two of those four components have a point in common, but their images do not, again see Figure~\ref{fig:proof}.

There are two cases remaining. $L, L'$ can have one component each, and then $L\cap L'$ has cardinality one. Otherwise, one of the cycles, say $L$, has two components, while $L'$ has only one, and then $L\cap L'$ has cardinality two.
\end{proof}

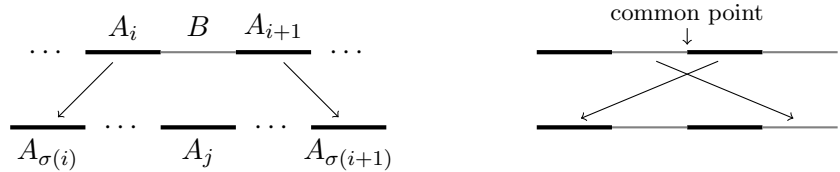
\begin{figure}[htb!!]
\begin{tikzpicture}
\draw [ultra thick] (1,1) -- (2,1) (3,1) -- (4,1);
\draw [thick, gray] (2,1) -- (3,1);
\node at (0.5,1) {$\cdots$};
\node at (4.5,1) {$\cdots$};
\node (b) at (2.5,1) {};
\node (ai) at (1.5,1) {};
\node (aip1) at (3.5,1) {};
\node [inner sep=0pt, minimum size=1pt, label=above:{$A_i$}] at (ai) {};
\node [inner sep=0pt, minimum size=1pt, label=above:{$A_{i+1}$}] at (aip1) {};
\node [inner sep=0pt, minimum size=1pt, label=above:{$B\vphantom{A_{i}}$}] at (b) {};
\draw[ultra thick] (0,0) -- (1,0) (2,0) -- (3,0) (4,0) -- (5,0);
\node at (1.5,0) {$\cdots$};
\node at (3.5,0) {$\cdots$};
\node (asi) at (0.5,0) {};
\node (aj) at (2.5,0) {};
\node (asip1) at (4.5,0) {};
\node [inner sep=0pt, minimum size=1pt, label=below:{$A_{\sigma(i)}$}] at (asi) {};
\node [inner sep=0pt, minimum size=1pt, label=below:{$A_j$}] at (aj) {};
\node [inner sep=0pt, minimum size=1pt, label=below:{$A_{\sigma(i+1)}$}] at (asip1) {};
\draw[->] (ai) -- (asi);
\draw[->] (aip1) -- (asip1);
\begin{scope}[shift={(7,0)}]
\draw [ultra thick] (0,1) -- (1,1) (2,1) -- (3,1) (0,0) -- (1,0) (2,0) -- (3,0);
\draw [thick, gray] (1,1) -- (2,1) (3,1) -- (4,1) (1,0) -- (2,0) (3,0) -- (4,0);
\node [inner sep=8pt] (b1) at (1.3,1) {};
\node [inner sep=8pt] (a2) at (2.7,1) {};
\node [inner sep=8pt] (a1) at (0.3,0) {};
\node [inner sep=8pt] (b2) at (3.7,0) {};
\draw [->] (b1) -- (b2);
\draw [->] (a2) -- (a1);
\node [inner sep=1pt] (lab) at (2,1.5) {\footnotesize common point};
\draw [->] (lab) -- (2,1.1);
\end{scope}
\end{tikzpicture}
\caption{Diagrams for the proof of Theorem~\ref{th:at-least-3}}
\label{fig:proof}
\end{figure}

In the course of the proof we have also shown the following result:

\begin{corollary}\label{cor:1-or-2}
If $f:[0,1]\to [0,1]$ has one or two points with full $\sa$-limit sets, but not more, then $[0,1]$ is the union of two transitive cycles of intervals.
\end{corollary}

Corollary~\ref{cor:1-or-2} corrects the second part of Conjecture 3 (originally \cite[Conjecture 4.14]{KMS}) to allow for two points with a full $\sa$-limit set%
\footnote{Incidentally, when there is exactly one point with a full $\sa$-limit set, the conclusion of the conjecture holds as stated in~\cite{KMS}: there is $c\in(0,1)$ such that such that $f|_{[0,c]}$ and $f|_{[c,1]}$ are both transitive.}%
. 
Both possibilities from the corollary are shown in Figure~\ref{fig:full}. One of the interval maps shown has exactly one point with a full $\sa$-limit set, and the other has exactly two.

\begin{figure}[htb!!]
\includegraphics[width=.2\textwidth]{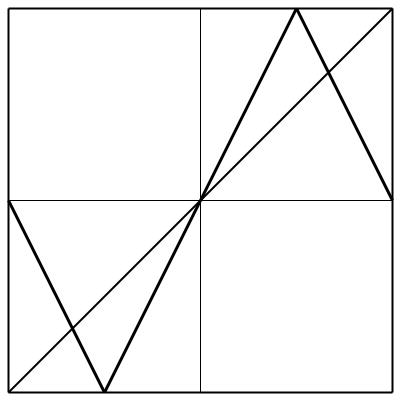}\hspace{.1\textwidth}\includegraphics[width=.2\textwidth]{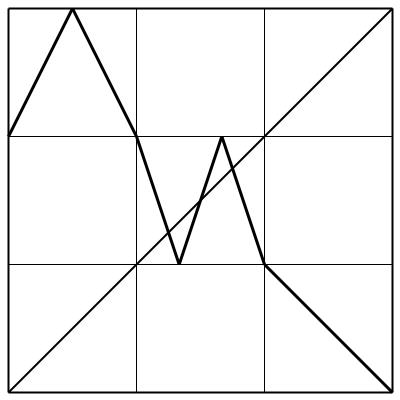}
\caption{Maps for which $\sa(x)=[0,1]$ for only 1 or 2 points $x$.}\label{fig:full}
\end{figure}

\subsection{Special $\alpha$-limit sets containing a common open set}\label{sec:common}\strut\\\indent
Now we study the $\sa$-limit sets that contain a given transitive cycle of intervals. We get a sharpening of Theorem~\ref{th:basic} in the case when $B=M$, i.e. when our basic set is itself a transitive cycle of intervals.

Let $M\subset[0,1]$ be a transitive cycle of intervals for $f$. For the reader's convenience, we recall some definitions from Section~\ref{sec:basic}. We write $\End(M)$ for the endpoints of the connected components of $M$ and refer to these points simply as \emph{endpoints of $M$}. The main role in our analysis is played by the set of \emph{exceptional points of $M$}
\begin{equation*}
E := M \setminus \bigcap_U \bigcup_{n=1}^\infty f^n(U),
\end{equation*}
where $U$ ranges over all relatively open nonempty subsets of $M$. It is known that $E$ is finite; it can contain some endpoints and at most one non-endpoint from each component of $M$. It is also known that $E$ and $M\setminus E$ are both invariant under $f$, see~\cite{Ruette}. Endpoints of $M$ in $E$ are called \emph{non-accessible endpoints}, as explained in Section~\ref{sec:basic}. From~\cite[Proposition 3.10]{KMS} it follows that $\sa(x)\supseteq M$ for all $x\in M\setminus E$. On the other hand, if $x\in E$, then its $\sa$-limit set is disjoint from the interior of $M\setminus E$ by Lemma~\ref{lem:inv-sa}. Therefore we have another characterization of $E$ using $\sa$-limit sets, $E=\{x\in M ~:~ \sa(x)\not\supseteq M\}$.

\begin{theorem}\label{th:at-most-3}
Let $M$ be a transitive cycle of intervals for $f:[0,1]\to[0,1]$ and let $E$ be its set of exceptional points.
\begin{enumerate}
\item Each point $x\in M \setminus \left( E \cup \End(M) \right)$ has the same $\sa$-limit set.
\item At most three distinct $\sa$-limit sets of $f$ contain $M$.
\end{enumerate}
\end{theorem}

We will see in the course of the proof that if $\sa(x')\supseteq M$ is distinct from the $\sa$-limit set described in part (1), then $x'$ belongs to a periodic orbit contained in $\End(M)$. Since there are at most two such periodic orbits in $\End(M)$ we get part (2).

Before giving the proof we discuss some consequences of this theorem.

\begin{corollary}\label{cor:at-most-3}
At most three distinct special alpha-limit sets of $f$ can contain a given nonempty open set.
\end{corollary}
\begin{proof}
Let $U$ be a nonempty open set in $[0,1]$. If a $\sa$-limit set of $f$ contains $U$, then by Theorem~\ref{th:M-nd}, $U$ contains a whole subinterval from some transitive cycle $M$. Applying Corollary~\ref{cor:tc}, we see that any $\sa$-limit set which contains $U$ also contains $M$. By Theorem~\ref{th:at-most-3} there are at most three such $\sa$-limit sets.
\end{proof}

This corollary corrects Conjecture~\ref{conj:4} (originally \cite[Conjecture 4.10]{KMS}), in which it was conjectured that two $\sa$-limit sets which contain a common open set must be equal. For comparison, note that if two $\omega$-limit sets of an interval map contain a common open set, then they are in fact equal, since an $\omega$-limit set with nonempty interior is itself a transitive cycle~\cite{Sharkovsky}. We also remark that the number three in Corollary~\ref{cor:at-most-3} cannot be improved, as is shown by the following example.

\begin{example}
Let $f:[0,1]\to[0,1]$ be an interval map for which $[0,\frac13]$ is a full 2-horseshoe, $[\frac13,\frac23]$ is a full 3-horseshoe, and $[\frac23,1]$ is a full 2-horseshoe, as shown in Figure~\ref{fig:overlap}. Then $\frac12$ belongs to only one of the three transitive invariant intervals for $f$ and $\sa(\frac12)=[\frac13,\frac23]$. But both $\frac13$ and $\frac23$ are accessible endpoints of adjacent transitive intervals and so $\sa(\frac13)=[0,\frac23]$ and $\sa(\frac23)=[\frac13,1]$.
\begin{figure}[htb!!]
\includegraphics[width=.2\textwidth]{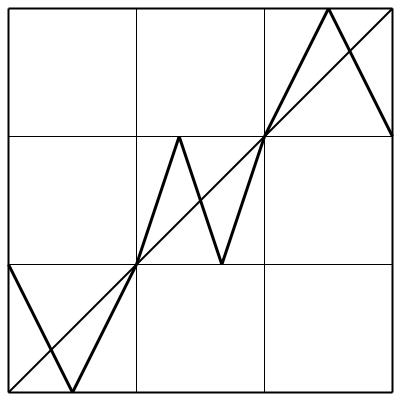}
\caption{An example where 3 distinct $\sa$-limit sets contain a common open interval.}\label{fig:overlap}
\end{figure}
\end{example}

In what follows it is necessary to allow for a weaker notion of a cycle of intervals for $f$. An interval is called \emph{non-degenerate} if it contains more than one point. If $U$ is a non-degenerate interval (not necessarily closed) such that $U, f(U), \ldots, f^{n-1}(U)$ are pairwise disjoint non-degenerate intervals and $f^n(U)\subseteq U$ (not necessarily equal), then we will call $\Orb(U)$ a \emph{weak cycle of intervals} of period $n$. The next lemma records one of the standard ways to produce a weak cycle of intervals. Similar lemmas appear in~\cite{Blokh} and several other papers, but since we were unable to find the exact statement we needed, we chose to give our own formulation here.

\begin{lemma}\label{lem:wc}
If a subinterval $U$ contains three distinct points from some orbit of $f$, then $\Orb(U)=\bigcup_{i=0}^\infty f^i(U)$ is a weak cycle of intervals for $f$.
\end{lemma}
\begin{proof}
Let $x, f^n(x), f^m(x)$ be three distinct points in $U$, $0<n<m$. Clearly $\Orb(U)$ is invariant. Since the intervals $U, f^n(U)$ both contain $f^n(x)$ we see that $U\cup f^n(U)$ is connected, i.e. it is an interval. Then also $f^n(U)\cup f^{2n}(U)$ is connected, and so on inductively. Therefore the set $A=\bigcup_{j=0}^\infty f^{jn}(U)$ is connected. Then $\Orb(U)=\bigcup_{i=0}^{n-1} f^i(A)$ has at most $n$ connected components. Let $B\supseteq A$ be the component of $\Orb(U)$ containing $U$, and let $k\leq n$ be minimal such that $B\cap f^k(B)\neq\emptyset$. Then $f^k(B)$ is a connected subset of $\Orb(U)$, so $f^k(B)\subseteq B$. For $0\leq i<j<k$ if $f^i(B)\cap f^j(B)\neq\emptyset$, then $f^{i+k-j}(B)\cap f^k(B)\neq\emptyset$, so $f^{i+k-j}(B)\cap B\neq\emptyset$, contradicting the choice of $k$. This shows that $B, f(B), \ldots, f^{k-1}(B)$ are pairwise disjoint. It remains to show that they are all non-degenerate. Clearly all three points $x,f^n(x),f^m(x)$ are in $B$. From the disjointness of $B,f(B),\ldots,f^{k-1}(B)$ it follows that $n,m$ are multiples of $k$. And since $x,f^n(x),f^m(x)$ are distinct we get $n=j_1k$, $m=j_2k$ with $0<j_1<j_2$. If any $f^i(B)$ is a singleton, $i\leq k$, then so also is $f^k(B)$. Then using $B\supseteq f^k(B)\supseteq f^{2k}(B) \supseteq \cdots$, we see that all the sets $f^{jk}(B)$, $j\geq1$, are the same singleton. But in that case $f^{j_1k}(B)=\{f^n(x)\}$, $f^{j_2k}(B)=\{f^m(x)\}$ are not the same singleton, which is a contradiction.
\end{proof}

\begin{lemma}\label{lem:cyclesmeet}
Let $N$ be a weak cycle of intervals for $f$ and $M$ a transitive cycle of intervals. Let $E$ be the set of exceptional points for $M$. If $N\cap M$ is nonempty, then the following hold:
\begin{enumerate}
\item Either $N\supseteq M\setminus E$ or else $N\cap M$ is a periodic orbit contained in $\End(M)$, and
\item The period of $N$ is at most twice the period of $M$.
\end{enumerate}
\end{lemma}
\begin{proof}
Let $n, m$ be the periods of $N, M$, respectively. Let $N_i$, $M_j$, be the components of $N$, $M$ with the temporal ordering, so $f(N_i) \subset N_{i+1\!\mod n}$ for all $i<n$ and $f(M_j) = M_{j+1\!\mod m}$ for all $j<m$.

Suppose first that $N\cap M$ is infinite. Then $N$ contains a non-degenerate interval $U\subset M$. We have $N\supseteq\Orb(U)\supseteq M\setminus E$, where the first containment comes from the invariance of $N$ and the second from the definition of the exceptional set $E$. Since $E$ contains at most one non-endpoint of each component $M_j$, it follow that each component $M_j$ meets at most two components $N_i$. So in this case $n\leq 2m$.

For the rest of the proof we suppose that $N\cap M$ is finite. We no longer need transitivity and the sets $N$, $M$ will play symmetric roles. Clearly each nonempty intersection $N_i \cap M_j$ is at common endpoints. This shows that $N\cap M\subseteq\End(M)$. It also shows that each component $N_i$ contains at most 2 points from $M$, and conversely each component $M_j$ contains at most 2 points from $N$.

\emph{Claim 1: Each component $N_i$ contains the same number of points of $M$. Conversely, each component $M_j$ contains the same number of points of $N$.} Suppose first that some $N_i$ contains 2 distinct points $a,b$ from $M$. Then $a,b$ belong to distinct components of $M$, and therefore $f(a), f(b)$ also belong to distinct components of $M$. But they both belong to $N_{i+1 \mod n}$. Continuing in this way we see that each each component of $N$ contains 2 points from $M$. Now suppose instead that each component of $N$ contains at most 1 point from $M$. Surely some $N_i$ contains at least one point $a\in M$. Then $N_{i+1 \mod n}$ contains the point $f(a)\in M$. Continuing in this way we see that each component $N_i$ contains exactly 1 point from $M$. Moreover, the whole argument still works if we reverse the roles of $N$ and $M$. This concludes the proof of Claim 1.

\emph{Claim 2: The intersection $N\cap M$ is a periodic orbit.} Suppose first that each component of $N$ contains 2 points from $M$. We reuse an argument from the proof of Theorem~\ref{th:at-least-3}. Let $A_1, A_2, \ldots, A_m$ be the components of $M$ in the spatial order, i.e. numbered from left to right in $[0,1]$. Let $\sigma:\{1,\ldots,m\} \to \{1,\ldots,m\}$ be the cyclic permutation defined by $f(A_i)\subseteq A_{\sigma(i)}$. If $m\geq3$ then there must exist $i$ such that $|\sigma(i)-\sigma(i+1)|\geq2$, so there is some $A_j$ strictly between $A_{\sigma(i)}, A_{\sigma(i+1)}$. Let $N_k$ be the component of $N$ which intersects both $A_i$ and $A_{i+1}$. By the intermediate value theorem, $f(N_k)\supset A_j$. This contradicts the fact that $M\cap N$ is finite. Therefore $M$ has only 2 components, and $N$ has only one, i.e. $N=N_1$. The two endpoints of $N_1$ belong to $A_1, A_2$, respectively, and are therefore interchanged by $f$. So in this case $M\cap N$ is a periodic orbit of period 2.

The symmetric situation arises if each component of $M$ contains 2 points from $N$. Then $N$ has 2 components and $M$ has only $1$, and again $M\cap N$ is a periodic orbit of period $2$.

Now suppose that each $N_i$ contains exactly one point from $M$, and each $M_j$ contains exactly one point from $N$. Then $m=n$ and we may assume the components are indexed such that $N_i\cap M_j$ is nonempty if and only if $i=j$. Let $x_i$ be the unique point of intersection of $N_i$ and $M_i$. Since these intersection points are unique we get $f(x_i)=x_{i+1 \mod n}$ for all $i$. Thus $N\cap M$ is a periodic orbit of period $m=n$.
\end{proof}

%

\begin{proof}[Proof of Theorem~\ref{th:at-most-3}]\strut\\\indent%
We have already seen that the set of exceptional points in $M$ may be characterized as $E=\{x\in M ~:~ \sa(x)\not\supseteq M\}$. Combined with Lemma~\ref{lem:inv-sa} this shows that
\begin{equation*}
\sa(x)\supseteq M \Leftrightarrow x\in M\setminus E.
\end{equation*}
Our task is to compare the $\sa$-limit sets of the various points $x\in M\setminus E$. Since they all contain $M$ it is enough to check whether or not they coincide outside of $M$. For any point $y\in[0,1]$ let us write $\saBasin(y)=\{z~:~y\in\sa(z)\}$.

\emph{Step 1: If $x\in M\setminus E$ and $y\in\sa(x)\setminus M$, then there is a weak cycle of intervals $N$ for $f$ such that $x\in N\subseteq\saBasin(y)$.}
To prove this claim, let $(x_i)$ be a backward orbit branch of $x$ accumulating on $y$. The sequence $(x_i)$ cannot contain $y$ twice, for otherwise $y$ would be a periodic point containing $x$ in its orbit, contradicting the invariance of $M$. Therefore there is a subsequence $(x_{i_j})$ which converges to $y$ monotonically from one side. We will suppose $x_{i_j} \searrow y$ from the right. The proof when $x_{i_j} \nearrow y$ from the left is similar.

For $k=1,2,\ldots$ let $U_k=(y,y+\frac1k)$. Then $U_k$ contains the points $x_{i_j}$ for large enough $j$. By Lemma~\ref{lem:wc}, the set $V_k=\Orb(U_k)$ is a weak cycle of intervals. Since $V_k$ is invariant and contains points $x_i$ for arbitrarily large natural numbers $i$, we must have the whole backward orbit branch $(x_i)$ contained in each $V_k$. Moreover, we have nesting $V_1\supseteq V_2\supseteq \cdots$. Letting $v_k$ denote the period (i.e. the number of connected components) of $V_k$ this implies $v_1\leq v_2\leq \cdots$. But each $V_k$ contains the point $x\in M$, so by Lemma~\ref{lem:cyclesmeet}~(2) each $v_k\leq2m$, where $m$ is the period of $M$. A bounded increasing sequence of natural numbers is eventually constant. So fix $k'$ such that $v_{k'}=v_{k'+1}=\cdots$.

For each $k$ let $V^0_k$ be the connected component of $V_k$ which contains $U_k$. Choose some $i$ such that $x_i\in V^0_{k'}$, $x_i>y$. For $k\geq k'$, $V_k$ is a weak cycle of intervals contained in $V_{k'}$ and with the same number of components. It follows that $V^0_k$ is the only component of $V_k$ which meets $V^0_{k'}$. Therefore $V^0_k$ must contain $x_i$ as well. In particular, setting $\delta=x_i-y$ we find that
\begin{equation*}
\forall_{\epsilon>0}, \Orb((y,y+\epsilon)) \supseteq (y,y+\delta).
\end{equation*}

Now let $N=\Orb((y,y+\delta))=\bigcap_{\epsilon>0}\Orb((y,y+\epsilon))$. It follows easily that $N\setminus\{y\}\subseteq\saBasin(y)$. For if $z\in N$, $z\neq y$, then taking $\epsilon_1<\min(\delta,|z-y|)$ we find $z_1\in(y,y+\epsilon_1)$ and $n_1\geq1$ such that $f^{n_1}(z_1)=z$. Then taking $\epsilon_2<z_1-y$ we find $z_2\in(y,y+\epsilon_2)$ and $n_2\geq1$ such that $f^{n_2}(z_2)=z_1$. Continuing inductively, we get a subsequence of a backward orbit branch of $z$ which accumulates on $y$.

Since $(y,y+\delta)$ contains $x_{i_j}$ for sufficiently large $j$, Lemma~\ref{lem:wc} also implies that $N$ is a weak cycle of intervals, and since it is forward invariant we have $x\in N$. This concludes Step 1.

\emph{Step 2: If $x, x'\in M\setminus E$ and $y\in\sa(x)\setminus M$, then $\sa(x)\subseteq\sa(x')$.} To prove this claim, fix an arbitrary point $y\in\sa(x)$. We need to show that $y\in\sa(x')$ as well. If $y\in M$ then there is nothing to prove, since $M\subset\sa(x')$. So suppose $y\notin M$. By Step 1 there is a weak cycle of intervals $N$ such that $x\in N\subseteq\saBasin(y)$. Now we apply Lemma~\ref{lem:cyclesmeet}~(1), noting that $M\cap N$ contains the point $x\not\in\End(M)$. Therefore $M\setminus E\subseteq N$. We now have $x'\in M\setminus E \subseteq N \subseteq \saBasin(y)$, from which it follows that $y\in\sa(x')$. This concludes Step 2.

\emph{Step 3: Each point in $M\setminus(E\cup\End(M))$ has the same $\sa$-limit set}. Suppose $x, x'\in M\setminus(E\cup\End(M))$. Then we may apply Step 2 to get both containments $\sa(x)\subseteq\sa(x')$ and $\sa(x')\subseteq\sa(x)$. This concludes Step 3 and the proof of Theorem~\ref{th:at-most-3}~(1). From now on we will refer to this common $\sa$-limit set as $S$.

\emph{Step 4: If the $\sa$-limit set of $x'$ contains $M$ and is distinct from $S$, then $x'$ belongs to a periodic orbit contained in $\End(M)$.} The hypothesis $M\subseteq\sa(x')$ implies that $x'\in M\setminus E$. Fix $x\in M\setminus(E\cup\End(M))$ and apply Step 2 to conclude that $S=\sa(x)\subseteq \sa(x')$. Our hypothesis is that this containment is strict. So choose $y\in\sa(x')\setminus\sa(x)$. Clearly $y\notin M$. Applying Step 1 to $x'$ we get a weak cycle of intervals $N$ such that $x'\in N \subseteq\saBasin(y)$. Moreover $x\not\in N$ since $y\not\in\sa(x)$. Now $M\cap N$ contains $x'$ but not $x$, so we may apply Lemma~\ref{lem:cyclesmeet}~(1) and conclude that $M\cap N$ is a periodic orbit contained in $\End(M)$. This concludes Step 4.

\emph{Step 5: At most three distinct $\sa$-limit sets of $f$ contain $M$.} One of these sets is $S$ from Step 3. By step 4, the only other sets to consider are the $\sa$-limit sets of those periodic points whose whole orbits are contained in $\End(M)$. Since $M$ is a cycle of intervals it is clear that $\End(M)$ contains at most two periodic orbits. And from the definitions it is clear that all points in a periodic orbit have the same $\sa$-limit set. This concludes Step 5 and the proof of Theorem~\ref{th:at-most-3}~(2).
\end{proof}

\section{On special $\alpha$-limit sets which are not closed}\label{sec:not-closed}

\subsection{Points in the closure of a special $\alpha$-limit set}\label{sec:points}\strut\\
We start Section~\ref{sec:not-closed} with two theorems which are important for understanding any non-closed $\sa$-limit sets of an interval map. Theorem~\ref{th:almost-closed} relates the closure of a $\sa$-limit set to the three kinds of maximal $\omega$-limit sets. Theorem~\ref{th:main-not-closed} determines precisely which points do or do not have a closed $\sa$-limit set, and which points from the closure are not present in the limit set. In Section~\ref{sec:not-closed-props} we apply these results to establish some topological properties of non-closed $\sa$-limit sets for interval maps, showing that they are always uncountable, nowhere dense, and of type $F_\sigma$ and $G_\delta$. In Section~\ref{sec:props} we address the question of which interval maps have all of their $\sa$-limit sets closed. Most importantly, this holds in the piecewise monotone case. In Section~\ref{sec:example} we give a concrete example of an interval map with a non-closed $\sa$-limit set.

Recall that a \emph{generating sequence} is any nested sequence of cycles of intervals $\Orb(I_0)\supset\Orb(I_1) \supset \cdots$ for an interval map $f$ with periods tending to infinity. In light of the discussion at the beginning of Section~\ref{sec:solenoidal} we may define a \emph{maximal solenoidal set} as the intersection $Q=\bigcap\Orb(I_n)$ of a generating sequence. Recall that a \emph{solenoidal $\omega$-limit set} is an infinite $\omega$-limit set containing no periodic points, and a solenoidal $\omega$-limit set is always contained in a maximal solenoidal set. Recall also that the \emph{Birkhoff center} of $f$ is the closure of the set of recurrent points.

\begin{theorem}\label{th:almost-closed}
Let $f$ be an interval map, let $x\in\overline{\sa(y)}$, and suppose that any of the following conditions holds:
\begin{enumerate}
\item $x$ is periodic,
\item $x$ belongs to a basic set $B$, or
\item $x$ is a recurrent point in a solenoidal $\omega$-limit set.
\end{enumerate}
Then $x\in\sa(y)$.
\end{theorem}

\begin{theorem}\label{th:main-not-closed}
Let $f:[0,1]\to [0,1]$ be an interval map and $y\in[0,1]$. The set $\sa(y)$ is not closed if and only if $y$ belongs to a maximal solenoidal set $Q$ which contains a nonrecurrent point from the Birkhoff center of $f$. In this case
\begin{equation*}
\overline{\sa(y)}\setminus\sa(y) = Q \cap \left(\overline{\Rec(f)}\setminus\Rec(f)\right).
\end{equation*}
\end{theorem}

The rest of Section~\ref{sec:points} is devoted to the proofs of these two theorems. For Theorem~\ref{th:almost-closed} the main idea is that $\overline{\sa(y)}\subset\alpha(y)$, so we can apply Theorems~\ref{th:solenoidal}, \ref{th:basic}, and \ref{th:periodic}. However, some extra care is needed for preperiodic points. Recall that $x$ is \emph{preperiodic} for $f$ if $x\not\in\Per(f)$ but there exists $n\geq1$ such that $f^n(x)\in\Per(f)$.

\begin{lemma}
Let $f$ be an interval map. If $x\neq f(x)=f^2(x)$ and $x\in\overline{\sa(y)}$, then $x\in\sa(y)$.
\end{lemma}
\begin{proof}
We may suppose without loss of generality that $x<f(x)$. A set $U$ is called a \emph{right-hand neighborhood} (resp. \emph{left-hand neighborhood}) of $x$ if it contains an interval $(x,x+\epsilon)$ (resp. $(x-\epsilon,x)$) for some $\epsilon>0$. By hypothesis either $x\in\sa(y)$, or every right-hand neighborhood of $x$ contains points from $\sa(y)$, or every left-hand neighborhood of $x$ contains points from $\sa(y)$. In the first case there is nothing to prove. The remaining two cases will be considered separately.

Suppose that every right-hand neighborhood of $x$ contains points from $\sa(y)$. We will construct inductively a sequence of points $y_n\to x$ and times $k_n>0$ such that $f^{k_0}(y_0)=y$ and $f^{k_n}(y_n)=y_{n-1}$ for $n\geq1$. We also construct points $a_n\in\sa(y)$ and open intervals $U_n\ni a_n$ that are compactly contained (i.e. their closure is contained) in $(x,f(x))$. The points $a_n$ will decrease monotonically to $x$, the intervals $U_n$ will be pairwise disjoint, and the inequalities $\sup U_{n} < y_{n-1} < \sup U_{n-1}$ will hold for all $n\geq1$.

For the base case, choose any point $a_0\in\sa(y)$ with $x<a_0<f(x)$. Choose a small open interval $U_0\ni a_0$ which is compactly contained in $(x,f(x))$. There exists $y_0\in U_0$ and $k_0>0$ such that $f^{k_0}(y_0)=y$. Now we make the induction step. Suppose that $x < y_{n-1} < \sup U_{n-1}$ and choose $a_n\in\sa(y)$ with $x<a_n<\min\{y_{n-1},\inf U_{n-1}, x+\frac1n\}$. Choose a small open interval $U_n\ni a_n$ compactly contained in $(x,f(x))$ with $\sup U_n < \min\{y_{n-1},\inf U_{n-1}\}$. By Theorem~\ref{th:structure} we know that $a_n$ is a non-wandering point, so there exist $b,c\in U_n$ and $k_n>0$ such that $f^{k_n}(b)=c$. Thus $f^{k_n}([x,b])$ contains both $f(x), c$, but $c<y_{n-1}<f(x)$, so by the intermediate value theorem there is $y_n\in(x,b)$ with $f^{k_n}(y_n)=y_{n-1}$. Clearly $y_n<\sup U_n$, so we are ready to repeat the induction step. This completes the proof in the case when every right-hand neighborhood of $x$ contains points from $\sa(y)$.

Now suppose that every left-hand neighborhood of $x$ contains points from $\sa(y)$. Consider the set
\begin{equation*}
W = \bigcap_{\epsilon>0} W_\epsilon, \text{ where } W_\epsilon = \bigcup_{t=1}^\infty f^t((x-\epsilon,x]).
\end{equation*}
For $\epsilon>0$ the set $W_\epsilon$ is invariant for $f$. It is connected because the intervals $f^t((x-\epsilon,x])$ all contain the common point $f(x)$. Now choose a point $a\in\sa(y)\cap(x-\epsilon,x)$. Since $a$ is non-wandering there are points $b,c\in(x-\epsilon,x)$ such that $f^t(b)=c$ for some $t>0$. Thus $c\in W_\epsilon$, so by connectedness we get $[x,f(x)]\in W_\epsilon$. Since these properties hold for arbitrary $\epsilon$, we see that $W$ is invariant, connected, and contains $[x,f(x)]$.

We claim that $W$ contains a left-hand neighborhood of $x$. Assume to the contrary that $x=\min W$. By continuity there is a point $z>f(x)$ such that $f([f(x),z]) \subset (x,1]$. Again by continuity there is a point $w<x$ such that $f([w,x]) \subset (x,z)$. Fix $\epsilon>0$ arbitrary. Put $U=(\max\{w,x-\epsilon\},x)$. Choose $a\in\sa(y)\cap U$. Since $a$ is non-wandering there are points $b,c\in U$ such that $f^s(b)=c$ for some $s>0$. But $f(b)>x$, so there is some $t\geq 1$ such that $f^t(b) > x$ and $f^{t+1}(b) < x$. Since $W$ is invariant, $f^t(b)\not\in W$. Therefore $f^t(b)>z$, so it follows that $z \in W_\epsilon$. Since $\epsilon>0$ was arbitrary, this shows that $z\in W$. Again fix $\epsilon>0$ arbitrary and let $U,a,b,c,s$ be as they were before. We have $f(b)\in(x,z)\subset W$ and $W$ is invariant, so $f^s(b)=c\in W$, contradicting the assumption that $x=\min W$.

Finally, we construct inductively a monotone increasing sequence of points $y_n \to x$ and a sequence of times $k_n>0$ such that $f^{k_0}(y_0)=y$ and $f^{k_n}(y_n)=y_{n-1}$ for $n\geq1$. Let $\delta>0$ be such that $W$ contains the left-hand neighborhood $(x-\delta,x)$. For the base case we find $a\in\sa(y)\cap (x-\delta,x)$. Then there is $y_0\in (x-\delta,x)$ and $k_0>0$ such that $f^{k_0}(y_0)=y$. For the induction step, suppose we are given $y_{n-1}\in (x-\delta,x)$. Choose a positive number $\epsilon<\min\{|x-y_n|,\frac1n\}$. Since $(0,\delta)\subset W \subset W_\epsilon$, we see from the definition of $W$ that there exist $y_n\in(x-\epsilon,x]$ and $k_n>0$ such that $f^{k_n}(y_n)=y_{n-1}$. Clearly $y_n\neq x$, since $y_{n-1}\neq f(x)$. Therefore $y_n\in(x-\delta,x)$ and the induction can continue.
\end{proof}

\begin{proposition}\label{prop:preperiodic}
Let $f$ be an interval map. If $x$ is preperiodic and $x\in\overline{\sa(y)}$, then $x\in\sa(y)$.
\end{proposition}
\begin{proof}
Find $n$ such that $x\neq f^n(x)=f^{2n}(x)$ and apply~\cite[Proposition 2.9]{KMS}, which says $sa(y,f)=\bigcup_{j=0}^{n-1} \sa(f^j(y),f^n)$. Since the closure of a finite union is the union of the closures, we find $j$ such that $x\in\overline{\sa(f^j(y),f^n)}$. By Lemma 1, $x\in\sa(f^j(y),f^n)$. Again applying~\cite[Proposition 2.9]{KMS} we conclude $x\in\sa(y,f)$.
\end{proof}

\begin{proof}[Proof of Theorem~\ref{th:almost-closed}]
Let $x\in\overline{\sa(y)}$. Since $\alpha(y)$ is a closed set containing $\sa(y)$, we have $x\in\alpha(y)$ as well. If $x$ is periodic, then by Theorem~\ref{th:periodic}, $x\in\sa(y)$. If $x$ is a recurrent point in a solenoidal $\omega$-limit set, then by Theorem~\ref{th:solenoidal} we again have $x\in\sa(y)$. If $x$ belongs to a basic set $B$, then it must be periodic, preperiodic, or have $\Orb(x)$ infinite. In the periodic case Theorem~\ref{th:periodic} applies. In the preperiodic case Proposition~\ref{prop:preperiodic} shows that $x\in\sa(y)$. And when $\Orb(x)$ is infinite, then it must be contained in both $B$ and $\alpha(y)$ since those sets are both invariant. So by Theorem~\ref{th:basic} we again have $x\in\sa(y)$.
\end{proof}

\begin{proof}[Proof of Theorem~\ref{th:main-not-closed}]
Suppose first that $\sa(y)$ is not closed. Pick any point $x\in\overline{\sa(y)}\setminus\sa(y)$. By Theorem~\ref{th:structure} we have $x\in\overline{\Rec(f)}$ and in particular $x\in\Lambda^1(f)$. Let $\omega(z)$ be a maximal $\omega$-limit set of $f$ containing $x$. By Theorem~\ref{th:almost-closed}, $x$ is not recurrent and $\omega(z)$ is a solenoidal $\omega$-limit set. Then $\omega(z)$ is contained in some maximal solenoidal set $Q$. By Theorem~\ref{th:solenoidal} we know that $y\in Q$. But $Q$ also contains $x$, so we have shown that $y$ belongs to a maximal solenoidal set which contains a nonrecurrent point from the Birkhoff center. Now if $x'$ is any other point from $\overline{\sa(y)}\setminus\sa(y)$, then by the same argument, $x'$ is a nonrecurrent point in the Birkhoff center and belongs to a maximal solenoidal set $Q'$ which also contains $y$. Since two maximal solenoidal sets are either disjoint or equal, we get $Q=Q'$. This shows that
\begin{equation*}
\overline{\sa(y)}\setminus\sa(y) \subseteq Q \cap \left(\overline{\Rec(f)}\setminus\Rec(f)\right).
\end{equation*}

To prove the converse part of the theorem, suppose that $Q$ is any maximal solenoidal set for $f$ which contains a nonrecurrent point from the Birkhoff center, and fix $y\in Q$. We will show that $\sa(y)$ is not closed, and that
\begin{equation*}
\overline{\sa(y)}\setminus\sa(y) \supseteq Q \cap \left(\overline{\Rec(f)}\setminus\Rec(f)\right).
\end{equation*}
To that end, let $x$ be any point from $Q\cap(\overline{\Rec(f)}\setminus\Rec(f))$. Choose a generating sequence $\Orb(I_0) \supset \Orb(I_1) \supset \cdots$ with $Q=\bigcap\Orb(I_n)$. Without loss of generality we may assume that $x\in I_n$ for each $n$. Since $x$ is not recurrent we have $x\not\in\sa(y)$ by Theorem~\ref{th:solenoidal}. To finish the proof it suffices to show that $x\in\overline{\sa(y)}$. We do so by finding periodic points arbitrarily close to $x$ which are in the $\alpha$-limit set of $y$, and then applying Theorem~\ref{th:periodic}.

Let $K$ be the connected component of $Q$ containing $x$. By Theorem~\ref{th:solenoid-model} we know that the singleton components of $Q$ are recurrent points, and in the non-degenerate components of $Q$, only the endpoints can appear in $\omega$-limit sets. Since $x$ is in an $\omega$-limit set but is not recurrent, we know that $K$ is a non-degenerate interval with $x$ as one of its endpoints. Without loss of generality we may assume that $x$ is the left endpoint of $K$ and write $K=[x,b]$.

Again using Theorem~\ref{th:solenoid-model} (and the fact that the set $S_{\Rec}$ referred to there is perfect) we know that each component of $Q$ has at least one recurrent endpoint, and the endpoint of a component is recurrent if and only if it is the limit of points of other components of $Q$. Since $x$ is not recurrent, we know that $b$ is. We conclude that $Q$ does not accumulate on $x$ from the left and it does accumulate on $b$ from the right.

Since $x$ is in the Birkhoff center $x\in\overline{\Rec(f)}=\overline{\Per(f)}$ and $K$ contains no periodic points, we know that the periodic points of $f$ accumulate on $x$ from the left. In particular, $x$ is not the left endpoint of $[0,1]$. That means we can find non-empty left-hand neighborhoods of $x$ in $[0,1]$, i.e. open intervals with right endpoint $x$.

Since $Q$ does not accumulate on $x$ from the left, we can find a left-hand neighborhood of $x$ which contains no points from $Q$, and then there must be some $n$ such that $I_n$ does not contain that left-hand neighborhood. It follows that $x=\min(Q\cap I_n)$. At this moment we do not know if $x$ is the left endpoint of $I_n$ or not, only that it is the left-most point of $Q$ in $I_n$.

Let $m$ be the period of the cycle of intervals $\Orb(I_n)$ and let $g=f^m$. Then $I_n$ is an invariant interval for $g$. Also $f$ and $g$ have the same periodic points (but not necessarily with the same periods). Since $g^i(x)\in I_n$ for all $i$ and $x\in Q$ is neither periodic nor preperiodic, there must be $i\leq2$ such that $g^i(x)$ is not an endpoint of $I_n$. By continuity there is $\delta>0$ such that $g^i((x-\delta,x])\subset I_n$. Then by invariance we get
\begin{equation}\label{star}\tag{*}
g^j((x-\delta,x])\subset I_n \text{ for all }j\geq i.
\end{equation}

Choose an arbitrary positive real number $\epsilon<\delta$. Choose a periodic point $p$ in $(x-\epsilon,x)$. Let $k$ be the period of $p$ under the map $g$. Then $g^{ik}(p)=p$ and by~\eqref{star} $g^{ik}(p)\in I_n$. This shows that $p\in I_n$. (Incidentally, we see that $x$ is not the left endpoint of $I_n$.)

The set of fixed points of $g^k$ is closed, so let $p'$ be the largest fixed point of $g^k$ in the interval $[p,x]$. Let $h$ be the map $h=g^k=f^{mk}$. We know that $h(x)$ is an element of $Q$ in $I_n$, and therefore it lies to the right of $x$. We also know $h(x)\notin[x,b]$ because no component of $Q$ returns to itself. So the order of our points is $p'=h(p')<x<b<h(x)$, and the graph of $h|_{(p',x]}$ lies above the diagonal, since there are no other fixed points there, see Figure~\ref{fig:graph-h}.

\begin{figure}[htb!!]
\begin{tikzpicture}[scale=0.5]
\draw (0,0) -- (9,0) -- (9,9) -- (0,9) -- cycle;
\draw (0,0) -- (9,9);
\draw [fill] (1,1) circle [radius=0.1] node [below right, inner sep=2] {\small $p$};
\draw [fill] (2,2) circle [radius=0.1] node [below right, inner sep=1] {\small $p'$};
\draw [fill] (3,3) circle [radius=0.06] node [below right, inner sep=2] {\small $x$};
\draw [fill] (4,4) circle [radius=0.06] node [below right, inner sep=2] {\small $b$};
\draw [fill] (7,7) circle [radius=0.06] node [below right, inner sep=2] {\small $h(x)$};
\draw [thick] (3,3) -- (4,4);
\draw [thick] (5,5) -- (6,6);
\node [below right, inner sep=2] at (5.5,5.5) {\small $L$};
\draw [dashed, thin] (3,3) -- (3,7) -- (7,7);
\draw [fill] (3,7) circle [radius=0.1];
\draw [thick] (2,2) to [out=70, in=250] (3,7);
\end{tikzpicture}
\caption{The graph of $h|_{I_n}$. Since $\Orb(L,f)\supset Q$ and each right-hand neighborhood of $p'$ eventually covers $L$, we see that $p'\in\alpha(y)$. Note: We do not claim that $h$ is monotone on $(p',x]$, only that the graph stays above the diagonal.}
\label{fig:graph-h}
\end{figure}
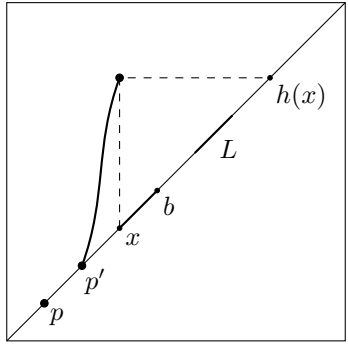

Since $Q$ accumulates on $b$ from the right we can find $n'>n$ such that one of the components $L$ of $\Orb(I_{n'}, f)$ lies between $b$ and $h(x)$. By the intermediate value theorem, $h([p',x])$ contains $L$.

Let $U$ be any right-hand neighborhood of $p'$, that is $U=(p',z)$ with $z-p'>0$ as small as we like. We consider the orbit of $U$. Since the graph of $h$ lies above the diagonal on $(p',x]$ we get a monotone increasing sequence $(h^{j}(z))_{j=0}^l$ with $l\geq1$ minimal such that $h^l(z)\geq x$ (if there is no such $l$, then $h^j(z)$ converges to a fixed point of $h$ in the interval $[z,x]$, which contradicts the choice of $p'$). Then $h^{l+1}(U)\supset L$. In particular, $\Orb(U, f) \supset \Orb(L, f) \supset \Orb(I_{n'}, f) \supset Q$.

We see that for every right-hand neighborhood $U$ of $p'$ there is a point $y'\in U$ and a time $t$ such that $f^t(y')=y$. As the neighborhood $U$ shrinks, the value of $t$ must grow without bound, because $p'$ is periodic. This shows that $p'\in\alpha(y)=\alpha(y,f)$. By Theorem~\ref{th:periodic} it follows that $p'\in\sa(y)$.

Since $p'$ was chosen from the $\epsilon$-neighborhood $(x-\epsilon,x)$ and $\epsilon>0$ can be arbitrarily small, we conclude that $x\in\overline{\sa(y)}$. This completes the proof.
\end{proof}

\subsection{Properties of a non-closed special $\alpha$-limit set}\label{sec:not-closed-props}\strut\\\indent
As an application of Theorems~\ref{th:main-not-closed} and~\ref{th:almost-closed}, we get the following results.

\begin{theorem}\label{th:not-closed}
If $\sa(y)$ is not closed, then it is uncountable and nowhere dense.
\end{theorem}

\begin{proof}
Suppose $\sa(y)$ is not closed. By Theorem~\ref{th:main-not-closed} we have $y\in Q$ for some solenoidal set $Q=\bigcap \Orb(I_n)$. By Theorem~\ref{th:solenoidal} we know that $\sa(y)$ contains $Q\cap\Rec(f)$. This set is perfect~\cite[Theorem 3.1]{Blokh}, and therefore uncountable.

Let $M$ be a transitive cycle of intervals for $f$. If $Q\cap M\neq\emptyset$, then by Lemma~\ref{lem:cyclesmeet}, each cycle of intervals $\Orb(I_n)$ has period at most twice the period of $M$. This contradicts the fact that the periods of the cycles of intervals tend to infinity. Therefore $Q$ does not intersect any transitive cycle of intervals $M$ for $f$. In particular, $y$ does not belong to any transitive cycle of intervals $M$. By Corollary~\ref{cor:tc} we conclude that $\sa(y)$ does not contain any transitive cycle. By Theorem~\ref{th:M-nd} it follows that $\sa(y)$ is nowhere dense.
\end{proof}

A subset of a compact metric space $X$ is called $F_\sigma$ if it is a countable union of closed sets, and $G_\delta$ if it is a countable intersection of open sets. These classes of sets make up the second level of the Borel hierarchy. Closed sets and open sets make up the first level of the Borel hierarchy and they are always both $F_\sigma$ and $G_\delta$. The next result shows that the $\sa$-limit sets of an interval map can never go past the second level of the Borel hierarchy in complexity.

\begin{theorem}
Each $\sa$-limit set for an interval map $f$ is both $F_\sigma$ and $G_\delta$.
\end{theorem}

\begin{proof}
We write $\Bas(f)$ for the union of all basic $\omega$-limit sets of $f$ and $\Sol(f)$ for the union of all solenoidal $\omega$-limit sets of $f$. We continue to write $\Per(f)$ for the union of all periodic orbits of $f$.

To prove that $\sa(y)$ is of type $F_\sigma$ we express it as the following union
\begin{equation*}
\sa(y)=\big(\sa(y)\cap\Per(f)\big) \cup \big(\sa(y)\cap\Bas(f)\big) \cup \big(\sa(y)\cap\Sol(f)\big),
\end{equation*}
and show that each of the three sets in the union is of type $F_\sigma$.

The set $\Per(f)=\bigcup_n\{x~:~f^n(x)=x\}$ is clearly of type $F_\sigma$. By Theorem~\ref{th:almost-closed}, $\sa(y)\cap\Per(f)$ is a relatively closed subset of $\Per(f)$, and is therefore of type $F_\sigma$.

Since an interval map has at most countably many basic sets~\cite[Lemma 5.2]{Blokh}, their union $\Bas(f)$ is of type $F_\sigma$. By Theorem~\ref{th:almost-closed}, we know that $\sa(y)\cap \Bas(f)$ is a relatively closed subset of $\Bas(f)$, and is therefore of type $F_\sigma$.

By Theorem~\ref{th:solenoidal} and Corollary~\ref{cor:at-most-1}, we know that $\sa(y)\cap\Sol(f)$ is either the empty set, or a single minimal solenoidal set $S$, and minimal solenoidal sets are closed. Closed sets are trivially of type $F_\sigma$.

To prove that $\sa(y)$ is of type $G_\delta$ it is enough to show that $\overline{\sa(y)}\setminus\sa(y)$ is at most countable. By Theorem~\ref{th:main-not-closed} we know that this set is of the form 
\begin{equation*}
\overline{\sa(y)}\setminus\sa(y) = Q\cap(\overline{\Rec(f)}\setminus\Rec(f))
\end{equation*}
for some maximal solenoidal set $Q$. But by Theorem~\ref{th:solenoid-model} the only points in $Q$ which can be in the Birkhoff center but not recurrent are endpoints of non-degenerate components of $Q$. Since $Q\subset[0,1]$ has at most countably many non-degenerate components, this completes the proof.
\end{proof}

\subsection{Maps which have all special $\alpha$-limit sets closed}\label{sec:props}\strut\\\indent
In~\cite{KMS} it was proved that all $\sa$-limit sets of $f$ are closed if $\Per(f)$ is closed. This is a very strong condition; it implies in particular that $f$ has zero topological entropy. In this section we give necessary and sufficient criteria to decide if all $\sa$-limit sets of $f$ are closed. We show in particular that this is the case for piecewise monotone maps.  Note that an interval map $f:[0,1]\to[0,1]$ is called \emph{piecewise monotone} if there are finitely many points $0=c_0 < c_1 < \cdots < c_n = 1$ such that for each $i<n$, the restriction $f|_{[c_i,c_{i+1}]}$ is monotone, i.e. non-increasing or non-decreasing.

\begin{theorem}\label{th:all-closed}
Let $f:[0,1]\to[0,1]$ be an interval map. The following are equivalent:
\begin{enumerate}
\item For some $y\in[0,1]$, $\sa(y)$ is not closed.
\item The attracting center $\Lambda^2(f)$ is not closed.
\item The attracting center is strictly contained in the Birkhoff center $\Lambda^2(f)\subsetneq\smash{\overline{\Rec(f)}}$.
\item Some solenoidal $\omega$-limit set of $f$ contains a non-recurrent point in the Birkhoff center.
\end{enumerate}
\end{theorem}

\begin{proof}[Proof of Theorem~\ref{th:all-closed}]
(1) $\implies$ (2): Suppose there is a point $y\in[0,1]$ with $\sa(y)$ not closed. Choose $x\in\overline{\sa(y)}\setminus\sa(y)$. By Theorem~\ref{th:main-not-closed} we know that $x$ is a non-recurrent point in a solenoidal set. By Theorem~\ref{th:solenoidal} it follows that no $\sa$-limit set of $f$ contains $x$. This shows that $x\in\overline{\SA(f)}\setminus\SA(f)$. Thus $\SA(f)$ is not closed. But $\Lambda^2(f)=\SA(f)$ by Theorem~\ref{th:structure}.

(2) $\implies$ (3): This follows immediately from the containments $\Rec(f)\subset\Lambda^2(f)\subset\overline{\Rec(f)}$ in Theorem~\ref{th:structure}.

(3) $\implies$ (4): Suppose $\SA(f)=\Lambda^2(f)\neq\overline{\Rec(f)}$. Then we can find $x\in\overline{\Rec(f)}\setminus\SA(f)$. By Theorem~\ref{th:structure} we know that $x\in\Lambda^1(f)$. Because $x$ is in an $\omega$-limit set, it must belong to a periodic orbit, a basic set, or a solenoidal $\omega$-limit set.

Each periodic orbit is contained in the $\sa$-limit set of any one of its points. Each basic set is also contained in a $\sa$-limit set by Corollary~\ref{cor:bs}. Since we supposed that $x$ is not in any $\sa$-limit set, we must conclude that $x$ is not in a periodic orbit or a basic set.

Now we know that $x$ belongs to a solenoidal $\omega$-limit set. Since $x$ is not in any $\sa$-limit set we may use Theorem~\ref{th:structure} to conclude that $x$ is not recurrent.

(4) $\implies$ (1): Suppose a solenoidal $\omega$-limit set $\omega(z)$ contains a non-recurrent point $x$ in the Birkhoff center $x\in\overline{\Rec(f)}$. Since $\omega(z)$ is solenoidal it has a generating sequence, i.e. a nested sequence of cycles of intervals $\Orb(I_0)\supset\Orb(I_1)\supset \cdots$ with $\omega(z)\subseteq Q=\bigcap_n\Orb(I_n)$. By Theorem~\ref{th:main-not-closed} for any $y\in Q$ the set $\sa(y)$ is not closed.
\end{proof}

\begin{corollary}
If $f$ is a piecewise monotone interval map, then all $\sa$-limit sets of $f$ are closed.
\end{corollary}
\begin{proof}
By~\cite[Lemma PM2]{Blokh} each point in a solenoidal $\omega$-limit set for a piecewise monotone map $f$ is recurrent, so condition (4) of Theorem~\ref{th:all-closed} can never be satisfied.
\end{proof}

We remark that in general the conditions of Theorem~\ref{th:all-closed} may be difficult to verify. Even condition (4) is difficult, since the non-recurrent points in a solenoidal $\omega$-limit set need not belong to the Birkhoff center. For an example, see~\cite{BlockCoven}. But for maps with zero topological entropy, the whole picture simplifies considerably. For the definitions of topological entropy and Li-Yorke chaos we refer the reader to any of the standard texts in topological dynamics, eg.~\cite{Ruette}.

\begin{corollary}\label{cor:all-closed}
Let $f:[0,1]\to[0,1]$ be an interval map with zero topological entropy. Then all $\sa$-limit sets of $f$ are closed if and only if $\Rec(f)$ is closed.
\end{corollary}
\begin{proof}
For an interval map $f$ with zero topological entropy, the set of recurrent points is equal to $\Lambda^2(f)$ by~\cite{Xiong}. Now apply Theorem~\ref{th:all-closed}.
\end{proof}

\begin{corollary}\label{cor:LiYorke}
Suppose $f:[0,1]\to[0,1]$ is not Li-Yorke chaotic. Then all $\sa$-limit sets of $f$ are closed.
\end{corollary}
\begin{proof}
When $f$ is not Li-Yorke chaotic, T.~H.~Steele showed that $\Rec(f)$ is closed~\cite[Corollary 3.4]{Steele}. Moreover, it is well known that such a map has zero topological entropy, see eg.~\cite{Ruette}. Now apply Corollary~\ref{cor:all-closed}.
\end{proof}

\subsection{Example of a non-closed special $\alpha$-limit set}\label{sec:example}\strut\\\indent
In 1986 Chu and Xiong constructed a map $f:[0,1]\to[0,1]$ with zero topological entropy such that $\Rec(f)$ is not closed~\cite{ChuXiong}. This example appeared 6 years before the definition of $\sa$-limit sets~\cite{Hero}, but by Corollary~\ref{cor:all-closed} it provides an example of an interval map whose $\sa$-limit sets are not all closed.

In this section we give a short direct proof that one of the $\sa$-limit sets of Chu and Xiong's map $f$ is not closed. Here are the key properties of the map $f$ from~\cite{ChuXiong}.
\begin{enumerate}
\item There is a nested sequence of cycles of intervals $[0,1]=M_0 \supset M_1 \supset M_2 \supset \cdots$ for $f$, where $M_n = \Orb(J_n)$ has period $2^n$.
\item For each $n\in\N$ the interval $J_n$ is the connected component of $M_n$ which appears farthest to the left in $[0,1]$.
\item For each $n\in\N$ we can express $J_n= A_n \cup J_{n+1} \cup B_n \cup K_{n+1} \cup C_n$ as a union of five closed non-degenerate intervals with disjoint interiors appearing from left to right in the order $A_n < J_{n+1} < B_n < K_{n+1} < C_n$.
\item For each $n\in\N$ the map $f^{2^n}:J_n\to J_n$ has the following properties:
\begin{itemize}
\item $f^{2^n}|_{A_n} : A_n \to A_n \cup J_{n+1} \cup B_n$ is an increasing linear bijection,
\item $f^{2^n}|_{J_{n+1}} : J_{n+1} \to K_{n+1}$ is surjective,
\item $f^{2^n}|_{B_n} : B_n \to K_{n+1} \cup B_n \cup J_{n+1}$ is a decreasing linear bijection,
\item $f^{2^n}|_{K_{n+1}} : K_{n+1} \to J_{n+1}$ is an increasing linear bijection, and
\item $f^{2^n}|_{C_n} : C_n \to B_n \cup K_{n+1} \cup C_n$ is an increasing linear bijection.
\end{itemize}
\item The nested intersection $J_\infty = \bigcap_{n=0}^\infty J_n$ is a non-degenerate interval $J_\infty=[x,y]$.
\end{enumerate}

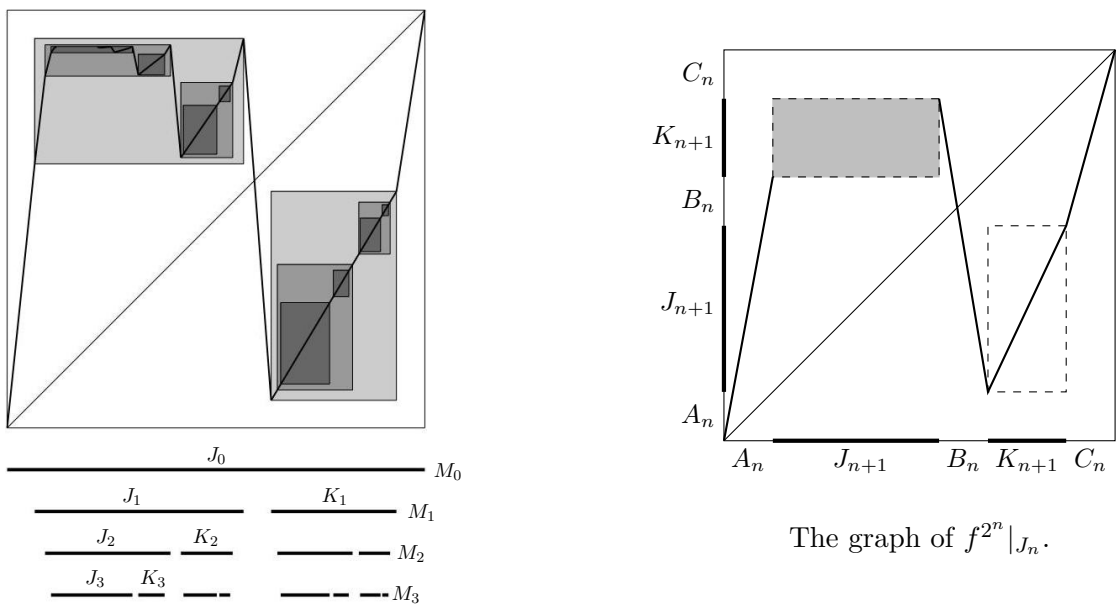
\begin{figure}[h!!]
\begin{minipage}{.4\textwidth}
\scalebox{0.7}{
\begin{tikzpicture}
\node [inner sep=0pt] at (0,0) {\includegraphics[width=8cm]{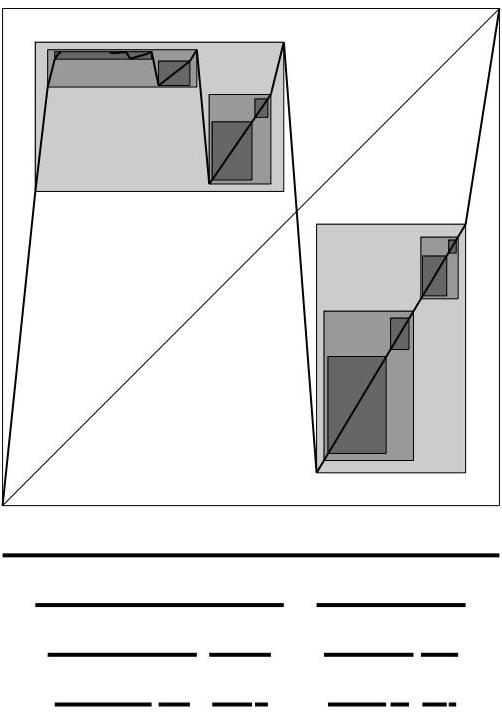}};
\node at (4.4cm,-3.2cm) {$M_0$};
\node at (3.9cm,-3.98cm) {$M_1$};
\node at (3.7cm,-4.76cm) {$M_2$};
\node at (3.6cm,-5.54cm) {$M_3$};
\node at (0cm,-2.90cm) {$J_0$};
\node at (-1.6cm,-3.68cm) {$J_1$};
\node at (2.25cm,-3.68cm) {$K_1$};
\node at (-2.1cm,-4.46cm) {$J_2$};
\node at (-0.2cm,-4.46cm) {$K_2$};
\node at (-2.3cm,-5.24cm) {$J_3$};
\node at (-1.2cm,-5.24cm) {$K_3$};
\end{tikzpicture}
}
\end{minipage}\hspace{.15\textwidth}%
\begin{minipage}{.4\textwidth}
\begin{tikzpicture}[scale=1.3]
\draw (0,0) -- (4,0) -- (4,4) -- (0,4) -- cycle;
\draw (0,0) -- (4,4);
\draw [fill, lightgray] (2.2,3.5) -- (0.5,3.5) -- (0.5,2.7) -- (2.2,2.7) -- cycle;
\draw [dashed] (3.5,2.2) -- (2.7,2.2) -- (2.7,0.5) -- (3.5,0.5) -- cycle;
\draw [dashed] (2.2,3.5) -- (0.5,3.5) -- (0.5,2.7) -- (2.2,2.7) -- cycle;
\draw [thick] (4,4) -- (3.5,2.2) -- (2.7,0.5) -- (2.2,3.5);
\draw [thick] (0.5,2.7) -- (0,0);
\draw [ultra thick] (3.5,0) -- (2.7,0);
\node [below, inner sep=3] at (3.1,0) {\small $K_{n+1}$};
\draw [ultra thick] (2.2,0) -- (0.5,0);
\node [below, inner sep=3] at (1.4,0) {\small $J_{n+1}$};
\draw [ultra thick] (0,3.5) -- (0,2.7);
\node [left, inner sep=3] at (0,3.1) {\small $K_{n+1}$};
\draw [ultra thick] (0,2.2) -- (0,0.5);
\node [left, inner sep=3] at (0,1.4) {\small $J_{n+1}$};
\node [below, inner sep=3] at (0.25,0) {\small $A_n$};
\node [below, inner sep=3] at (2.45,0) {\small $B_n$};
\node [below, inner sep=3] at (3.75,0) {\small $C_n$};
\node [left, inner sep=3] at (0,0.25) {\small $A_n$};
\node [left, inner sep=3] at (0,2.45) {\small $B_n$};
\node [left, inner sep=3] at (0,3.75) {\small $C_n$};
\node at (2,-1) {The graph of $f^{2^n}|_{J_n}$.};
\end{tikzpicture}
\end{minipage}
\caption{A map with a $\sa$-limit set which is not closed.}\label{fig:map}
\end{figure}

The graph of the map $f:[0,1]\to[0,1]$ is shown in Figure~\ref{fig:map}. Chu and Xiong showed that the left endpoint $x$ of $J_\infty$ is not recurrent, but it is a limit of recurrent points. We give a short direct proof that $\sa(x)$ is not closed.

\begin{theorem}
Let $f$ be the interval map defined in~\cite{ChuXiong} and let $x$ be the left endpoint of $J_\infty$ as defined above. Then $x\in\overline{\sa(x)}\setminus\sa(x)$, and therefore $\sa(x)$ is not closed.
\end{theorem}
\begin{proof}
Fix $n\in\N$ and let $a_n$ be the left endpoint of the interval $A_n$. Property (4) tells us that $f^{2^n}:A_n \to A_n\cup J_{n+1}\cup B_n$ is linear, say, with slope $\lambda_n$. By property (5) we have $x\in J_{n+1}$. Therefore there is a backward orbit branch $\{x_i\}_{i=0}^\infty$ of $x$ such that $x_{k\cdot 2^n} = a_n+\frac{x-a_n}{\lambda_n^k}$ for all $k\in\N$. This shows that $a_n\in\sa(x)$. But $n\in\N$ was arbitrary. If we let $n\to\infty$, then $a_n\to x$. This shows that $x\in\overline{\sa(x)}$.

Now we will show that $x\not\in\sa(x)$. Let $\{x_i\}_{i=0}^\infty$ be any backward orbit branch of $x$. Let $Q=\bigcap_{n=0}^\infty M_n$. We distinguish two cases. First suppose that $x_i\in Q$ for all $i$. For any given $i\geq1$ we can choose $n$ with $2^n>i$. Since $x_i\in M_n$, $f^i(x_i)\in J_n$, and $M_n$ is a cycle of intervals of period $2^n$, we know that $x_i\not\in J_n$. Since $J_n$ is the left-most component of $M_n$ in $[0,1]$, it follows that $x_i>y>x$ (recall that $J_\infty=\bigcap J_n=[x,y]$ is non-degenerate). Since this holds for all $i\geq1$ we see that the backward orbit branch does not accumulate on $x$.

Now suppose there is $i_0$ with $x_{i_0}\not\in Q$. For each $i\geq i_0$ there is $n(i)\in \N$ such that $x_i\in M_{n(i)}\setminus M_{n(i)+1}$. Since $f(x_{i+1})=x_i$ and each $M_n$ is invariant, we get $n(i+1)\leq n(i)$. A non-increasing sequence of natural numbers must eventually reach a minimum, say, $n(i_1)=n(i_1+1)=\cdots = n$. For $i\geq i_1$, $x_i\not\in M_{n+1}$, so in particular $x_i\not\in J_{n+1}$. But by properties (3) and (5) we know that $J_{n+1}$ is a neighborhood of $x$. This shows that this backward orbit branch does not accumulate on $x$ either.
\end{proof}

\section{Open Problems}\label{sec:open-problems}

Only one problem concerning $\sa$-limit sets of interval maps in \cite{KMS} remains open:

\begin{problem}\label{prob:A}\cite{KMS}
Characterize all subsets $A$ of [0,1] for which there exists a continuous map $f:[0,1]\rightarrow [0,1]$ and a point $x\in [0,1]$ such that $\sa(x,f)=A$.
\end{problem}

We have seen that even for interval maps, $\sa$-limit sets need not be closed. If we want to work with closed limit sets, then there are several possible solutions. The first one, suggested to us by \v{L}.~Snoha, is to answer the following question: \emph{What are some sufficient conditions on a topological dynamical system $(X,f)$ so that all of its $\sa$-limit sets are closed?} In this regard we state one conjecture which we were not able to resolve.

\begin{conjecture}
If $f:[0,1]\to[0,1]$ is continuously differentiable, then all $\sa$-limit sets of $f$ are closed.
\end{conjecture}

Another possibility is to ask whether the ``typical'' interval map has all $\sa$-limit sets closed. Let $C^0([0,1])$ be the complete metric space of all maps $f:[0,1] \to [0,1]$ with the usual uniform metric $d(f,g) = \sup_{x\in[0,1]} |f(x)-g(x)|$. If some comeager subset of maps in $C^0([0,1])$ all have some property, then we call that property \emph{generic}.

\begin{problem}
Is the property of having all $\sa$-limit sets closed a generic property in $C^0([0,1])$?
\end{problem}

Another possible solution is to work with the closures of $\sa$-limit sets. Therefore we propose the following definition.

\begin{definition}
Let $(X,f)$ be a discrete dynamical system (i.e.~a continuous self-map on a compact metric space) and let $x\in X$. The \emph{$\beta$-limit set of $x$}, denoted $\beta(x)$ or $\beta(x,f)$, is the smallest closed set such that $d(x_n,\beta(x)) \to 0$ as $n\to\infty$ for every backward orbit branch $\{x_n\}_{n=0}^\infty$ of the point $x$.
\end{definition}

The letter $\beta$ here means ``backward,'' since $\beta$-limit sets serve as attractors for backward orbit branches. It is clear from the definition that $\beta(x)=\overline{\sa(x)}$.

\begin{proposition}
If $(X,f)$ is a discrete dynamical system and $x\in X$, then $\beta(x)$ is closed and strongly invariant, i.e. $f(\beta(x))=\beta(x)$. Additionally, $\beta(x)$ is nonempty if and only if $x\in\bigcap_{n=0}^\infty f^n(X)$. In particular, $\beta(x)$ is nonempty for every $x\in X$ when $f$ is surjective.
\end{proposition}
\begin{proof}
By~\cite[Corollary 2.7]{KMS} we have $f(\sa(x))=\sa(x)$. Taking closures, we get $f(\beta(x))=\beta(x)$. Additionally, $\beta(x)$ is nonempty if and only if $\sa(x)$ is. But by~\cite[Proposition 2.3]{KMS} we have $\sa(x)\neq\emptyset$ if and only if $x\in\bigcap_{n=0}^\infty f^n(X)$.
\end{proof}

We can simplify Problem~\ref{prob:A} by working with $\beta$-limit sets, since they are always closed.

\begin{problem}
Characterize all subsets $A \subseteq [0,1]$ for which there exists a continuous map $f:[0,1]\to [0,1]$ and a point $x\in [0,1]$ such that $\beta(x,f)=A$.
\end{problem}

If $X$ is a compact metric space, then the space $K(X)$ consisting of all nonempty compact subsets of $X$ can be topologized with the Hausdorff metric and it is again compact. The map $X\ni x\mapsto\omega(x,f)\in K(X)$ associated to a dynamical system $(X,f)$ is usually far from continuous, but its Baire class can be useful, see eg.~\cite{Steele}. Therefore we propose the following question.

\begin{problem}
Of what Baire class (if any) is the function $[0,1]\ni x\mapsto\beta(x,f)\in K([0,1])$ when $f$ is a surjective interval map?
\end{problem}

Hero used $\sa$-limit sets to characterize the attracting center $\Lambda^2(f)$ of an interval map $f$, and some work has been done to extend his results to trees and graphs~\cite{Hero, SXCZ, SXL}. We conjecture that for graph maps, the $\beta$-limit sets can be used to characterize the Birkhoff center $\overline{\Rec(f)}$ as follows:

\begin{conjecture}
Let $f:X\to X$ be a graph map, and let $x\in X$. The following are equivalent:
\begin{enumerate}
\item $x\in\overline{\Rec(f)}$
\item $x\in\beta(x)$
\item There exists $y\in X$ such that $x\in\beta(y)$.
\end{enumerate}
\end{conjecture}

The coexistence of periodic orbits for interval maps was studied by Sharkovsky~\cite{Shar64}. A special case of his theorem, proved independently by Li and Yorke, says that if $f:[0,1]\to [0,1]$ has a periodic orbit of period three, then it has periodic orbits of all periods~\cite{LiYorke}.

This suggests the problem of studying the coexistence of periodic orbits within special $\alpha$-limit sets\footnote{or equivalently, within $\beta$-limit sets, since by Theorem~\ref{th:periodic} a periodic orbit of an interval map $f$ is contained in $\sa(x)$ if and only if it is contained in $\beta(x)$.}. In this spirit, we offer one conjecture and one open problem.

\begin{conjecture}\label{conj:Sharkovsky}
Let $f:[0,1]\to[0,1]$ and $x\in[0,1]$. If $\sa(x)$ contains a periodic orbit of period 3, then for every positive integer $n$, $\sa(x)$ contains a periodic orbit of period $n$ or $2n$.
\end{conjecture}

\begin{problem}
For which subsets $A\subseteq\mathbb{N}$ is there a map $f:[0,1]\to[0,1]$ and a point $x\in[0,1]$ such that $A$ is the set of periods of all periodic orbits of $f$ contained in $\sa(x)$?
\end{problem}

In support of Conjecture~\ref{conj:Sharkovsky}, we show that the conclusion holds for $n=1$.

\begin{lemma}
Let $f:[0,1] \to [0,1]$ and $x\in[0,1]$. If $\sa(x)$ contains a periodic orbit of period 3, then it contains a periodic orbit of period 1 or 2 as well.
\end{lemma}
\begin{proof}
Let $\sa(x)$ contain the periodic orbit $\{a,b,c\}$ for the interval map $f$ with $a<b<c$. We may assume without loss of generality that $f(a)=b$, $f(b)=c$, and $f(c)=a$. By continuity we may choose a closed interval $U=[b-\epsilon,b+\epsilon]$ with $\epsilon>0$ small enough that $f^2(U)<U<f(U)$, that is to say, $\max f^2(U) < \min U < \max U < \min f(U)$. Now find $x_1\in U$ and $n_1\geq1$ such that $f^{n_1}(x_1)=x$. By the intermediate value theorem we can find $x_2\in (\max U, \min f(U))$ such that $f(x_2)=x_1$. Again, by the intermediate value theorem we can find $x_3\in (x_1,x_2)$ such that $f(x_3)=x_2$. In the next step we find $x_4\in (x_3,x_2)$ such that $f(x_4)=x_3$. Continuing inductively we find a whole sequence $(x_i)$ such that $f(x_{i+1})=x_i$, $i\geq 1$, arranged in the following order,
\begin{equation*}
x_1<x_3<x_5<\cdots < \cdots < x_6 < x_4 < x_2.
\end{equation*}
Since a bounded monotone sequence of real numbers has a limit, we may put $x_\infty^-=\lim_{i\to\infty} x_{2i+1}$ and $x_\infty^+=\lim_{i\to\infty} x_{2i+2}$, and we have $x_\infty^- \leq x_\infty^+$. Then $f(x_\infty^-)=\lim f(x_{2i+1}) = \lim x_{2i} = x_\infty^+$ and similarly $f(x_\infty^+)=x_\infty^-$. This shows that $\{x_\infty^+,x_\infty^-\}$ is a periodic orbit contained in $\sa(x)$. The period is 2 if these points are distinct and 1 if they coincide.
\end{proof}

\begin{example}
Let $f:[0,5]\to[0,5]$ be the ``connect-the-dots'' map with $f(0)=1$, $f(1)=5$, $f(4)=2$, $f(5)=0$, and which is linear (affine) on each of the intervals $[0,1]$, $[1,4]$, and $[4,5]$. Then $\sa(0)$ contains the period-three orbit $\{0,1,5\}$ and the period-two orbit $\{2,4\}$, but not the unique period-one orbit $\{3\}$.
\end{example}

\begin{example}
Let $f:[0,8]\to[0,8]$ be the ``connect-the-dots'' map with $f(0)=4$, $f(4)=8$, $f(5)=3$, $f(8)=0$, and which is linear (affine) on each of the intervals $[0,4]$, $[4,5]$, and $[5,8]$. Then $\sa(0)$ contains the period-three orbit $\{0,4,8\}$, the period-four orbit $\{1,5,3,7\}$ and the period-one orbit $\{\frac{14}{3}\}$, but not the unique period-two orbit $\{2,6\}$.
\end{example}

\addtocontents{toc}{\protect\setcounter{tocdepth}{0}}
\subsection*{Acknowledgements}
We wish to thank Piotr Oprocha for pointing us in the right direction at the beginning of this project, and telling us about several useful ideas from~\cite{Bal}. We also thank \v{L}ubo Snoha for his helpful feedback and warm encouragement. His insightful suggestions resulted in a much clearer paper, and in particular Theorem~\ref{th:main-not-closed} would not exist without him.\\

\begin{table}[h]

\begin{tabular}[t]{b{1.5cm} m{13.5cm}}

\includegraphics [width=.09\textwidth]{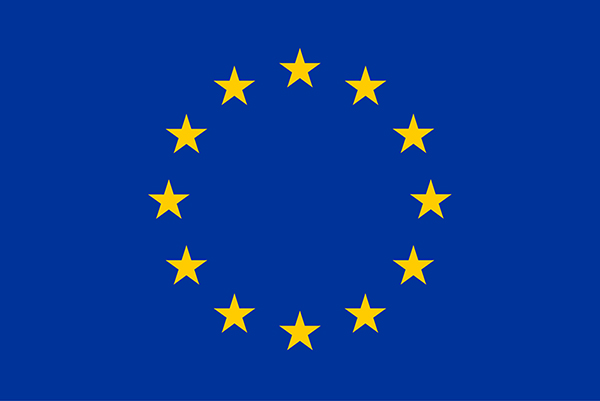} & 
This research is part of a project that has received funding from the European Union's Horizon 2020 research and innovation programme under the Marie Sk\l odowska-Curie grant agreement No 883748.

\end{tabular}
\end{table}

\end{document}